\documentclass{amsart}
\usepackage{amsmath, amsthm, amssymb, amscd, epsfig, eucal}
\usepackage{pinlabel}

\begin{document}

\newtheorem{theorem}{Theorem}[subsection]
\newtheorem{lemma}[theorem]{Lemma}
\newtheorem{proposition}[theorem]{Proposition}
\newtheorem{corollary}[theorem]{Corollary}
\newtheorem{conjecture}[theorem]{Conjecture}
\newtheorem{question}[theorem]{Question}
\newtheorem{problem}[theorem]{Problem}
\newtheorem*{claim}{Claim}
\newtheorem*{criterion}{Criterion}

\theoremstyle{definition}
\newtheorem{definition}[theorem]{Definition}
\newtheorem{construction}[theorem]{Construction}
\newtheorem{notation}[theorem]{Notation}

\theoremstyle{remark}
\newtheorem{remark}[theorem]{Remark}
\newtheorem{example}[theorem]{Example}

\numberwithin{equation}{subsection}

\newcommand\id{\textnormal{id}}

\newcommand\M{\mathcal M}
\newcommand\Z{\mathbb Z}
\newcommand\R{\mathbb R}
\newcommand\C{\mathbb C}
\newcommand\RP{\mathbb{RP}}
\newcommand\Q{\mathbb Q}
\newcommand\N{\mathbb N}
\renewcommand\H{\mathbb H}
\newcommand\Haus{H}
\newcommand\QI{\textnormal{QI}}
\newcommand\CAT{\textnormal{CAT}}
\newcommand\Out{\textnormal{Out}}

\renewcommand\Pr{{\mathbb{P}}}
\newcommand\Ex{{\mathbf E}}
\newcommand\Pee{{\mathbf P}}

\newcommand\cl{\textnormal{cl}}
\newcommand\scl{\textnormal{scl}}
\newcommand{\con}[1]{C_{#1}}
\newcommand{\length}{\textnormal{length}}
\newcommand{\eval}{\textnormal{eval}}
\newcommand{\word}{\textnormal{word}}
\newcommand{\path}{\textnormal{path}}
\newcommand{\cone}{\textnormal{cone}}
\newcommand{\bnd}{\textnormal{bnd }}

\newcommand\til{\widetilde}
\newcommand\Tri{\Delta}
\newcommand\SL{\textnormal{SL}}
\newcommand\1{{\bf 1}}

\newcommand{\norm}[1]{\left|#1\right|}

\title{The ergodic theory of hyperbolic groups}
\author{Danny Calegari}
\address{Department of Mathematics \\ Caltech \\
Pasadena CA, 91125}
\email{dannyc@its.caltech.edu}
\date{\today}

\begin{abstract}
These notes are a self-contained introduction to the use of dynamical and probabilistic 
methods in the study of hyperbolic groups. Most of this material is standard;
however some of the proofs given are new, and some results are proved in greater
generality than have appeared in the literature.
\end{abstract}

\maketitle

\tableofcontents

\section{Introduction}

These are notes from a minicourse given at a workshop in Melbourne July 11--15 2011.
There is little pretension to originality; the main novelty is firstly that we give
a new (and much shorter) proof of Coornaert's theorem on Patterson--Sullivan
measures for hyperbolic groups (Theorem~\ref{measure_is_conformal}), and secondly that
we explain how to combine the results of Calegari--Fujiwara in \cite{Calegari_Fujiwara}
with that of Pollicott--Sharp \cite{Pollicott_Sharp} to prove central limit theorems
for quite general classes of functions on hyperbolic groups
(Corollary~\ref{distance_CLT} and Theorem~\ref{holder_CLT}), crucially {\em without}
the hypothesis that the Markov graph encoding an automatic structure is ergodic.

A final section on random walks is much more cursory.

\section{Hyperbolic groups}

\subsection{Coarse geometry}

The fundamental idea in geometric group theory is to study groups as automorphisms
of geometric spaces, and as a special case, to study the group itself (with its
canonical self-action) as a geometric space. This is accomplished most directly by
means of the {\em Cayley graph} construction.

\begin{definition}[Cayley graph]
Let $G$ be a group and $S$ a (usually finite) generating set. Associated to $G$ and $S$
we can form the {\em Cayley graph} $C_S(G)$. This is a graph with vertex set $G$,
and with an edge from $g$ to $gs$ for all $g\in G$ and $s\in S$.
\end{definition}

The action of $G$ on itself by (left) multiplication induces a properly discontinuous 
action of $G$ on $C_S(G)$ by simplicial automorphisms. 

If $G$ has no $2$-torsion, the action is free and properly discontinuous, and the
quotient is a wedge of $|S|$ circles $X_S$. In this case, if
$G$ has a presentation $G=\langle S\; | \; R \rangle$ we can think of $C_S(G)$ as
the covering space of $X_S$ corresponding to the subgroup of the
free group $F_S$ normally generated by $R$, and the action of $G$ on $C_S(G)$ is
the deck group of the covering.

\begin{figure}[ht]
\centering
\labellist
\small\hair2pt
\endlabellist
\includegraphics[scale=1]{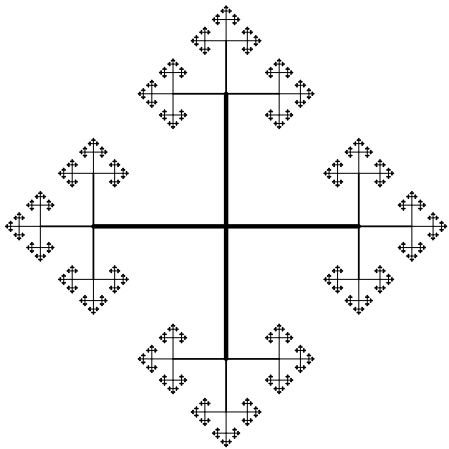}
\caption{The Cayley graph of $F_2=\langle a,b \; | \; \rangle$ with generating
set $S=\lbrace a,b\rbrace$}
\label{fig:F2_cayley}
\end{figure}
\medskip

We assume the reader is familiar with the notion of a metric space, i.e.\/ a space
$X$ together with a symmetric non-negative real-valued function $d_X$ on $X\times X$ 
which vanishes precisely on the diagonal, and which satisfies the triangle inequality
$d_X(x,y)+ d_X(y,z) \ge d_X(x,z)$ for each triple $x,y,z\in X$. 
A metric space is a {\em path metric space} if for each $x,y\in X$, the distance
$d_X(x,y)$ is equal to the infimum of the set of numbers $L$ for which there is
a $1$-Lipschitz map $\gamma:[0,L] \to X$ sending $0$ to $x$ and $L$ to $y$. It is
a {\em geodesic} metric space if it is a path metric space and if the infimum is achieved
on some $\gamma$ for each pair $x,y$; such a $\gamma$ is called a geodesic. Finally,
a metric space is {\em proper} if closed metric balls of bounded radius
are {\em compact} (equivalently, for each point $x$ the function $d(x,\cdot):X \to \R$
is proper).

The graph $C_S(G)$ can be canonically equipped with the structure of a 
{\em geodesic metric space}. This is accomplished by making each edge 
isometric to the Euclidean unit interval. If $S$ is finite, $C_S(G)$ is {\em proper}. 
Note that $G$ itself inherits a subspace
metric from $C_S(G)$, called the {\em word metric}. We denote the word metric by
$d_S$, and define $|g|_S$ (or just $|g|$ if $S$ is understood) to be $d_S(\id,g)$. 
Observe that $d_S(g,h)=|g^{-1}h|_S = |h^{-1}g|_S$ and that
$|g|_S$ is the length of the shortest word in elements of $S$ and their inverses 
representing the element $g$.

\medskip

The most serious shortcoming of this construction is its dependence on the choice of
a generating set $S$. Different choices of generating set $S$ give rise to different
spaces $C_S(G)$ which are typically not even homeomorphic. The standard way to resolve
this issue is to coarsen the geometric category in which one works.

\begin{definition}
Let $X,d_X$ and $Y,d_Y$ be metric spaces. A map $f:X \to Y$ ({\em not} assumed to be
continuous) is a {\em quasi-isometric map} if there are constants $K\ge 1,\epsilon\ge 0$ 
so that
$$K^{-1}d_X(x_1,x_2) - \epsilon \le d_Y(f(x_1),f(x_2)) \le Kd_X(x_1,x_2) + \epsilon$$
for all $x_1,x_2\in X$. It is said to be a {\em quasi-isometry} if further $f(X)$ is
a {\em net} in $Y$; that is, if there is some $R$ so that $Y$ is equal
to the $R$-neighborhood of $f(X)$.
\end{definition}
One also uses the terminology {\em $K,\epsilon$ quasi-isometric map} or
{\em $K,\epsilon$ quasi-isometry} if the constants are specified. Note that 
a $K,0$ quasi-isometric map is the same thing as a $K$ bilipschitz map.
The best constant $K$ is called the {\em multiplicative constant}, and the best
$\epsilon$ the {\em additive constant} of the map.

We denote the $R$-neighborhood of a set $\Sigma$ by $N_R(\Sigma)$. Hence
a quasi-isometry is a quasi-isometric map for which $Y=N_R(f(X))$ for some $R$.

\begin{remark}
It is much more common to use the terminology {\em quasi-isometric embedding} instead
of quasi-isometric map as above; we consider this terminology misleading, and therefore
avoid it.
\end{remark}

\begin{lemma}\label{qi_equivalence}
Quasi-isometry is an equivalence relation.
\end{lemma}
\begin{proof}
Reflexivity and transitivity are obvious, so we must show symmetry.
For each $y\in Y$ choose $x\in X$ with $d_Y(y,f(x))\le R$ (such an $x$ exists
by definition) and define $g(y)=x$. Observe $d_Y(y,fg(y))\le R$ by definition. Then
$$d_X(g(y_1),g(y_2))\le Kd_Y(fg(y_1),fg(y_2)) + K\epsilon \le Kd_Y(y_1,y_2) + K(\epsilon+2R)$$
Similarly,
$$d_X(g(y_1),g(y_2))\ge K^{-1}d_Y(fg(y_1),fg(y_2)) - K^{-1}\epsilon \ge K^{-1}d_Y(y_1,y_2) - K^{-1}(\epsilon+2R)$$
proving symmetry.
\end{proof}
Note that the compositions $fg$ and $gf$ as above move points a bounded distance.
One can define a category in which objects are equivalence classes of metric spaces
under the equivalence relation generated by thickening (i.e.\/ isometric inclusion as
a net in a bigger space), and morphisms are equivalence classes of quasi-isometric
maps, where two maps are equivalent if their values on each point are a uniformly
bounded distance apart. In this category, quasi-isometries are isomorphisms. In particular, the
set of quasi-isometries of a metric space $X$, modulo maps that move points a bounded distance,
is a {\em group}, denoted $\QI(X)$, which only depends on the quasi-isometry type of $X$. 
Determining $\QI(X)$, even for very simple spaces, is typically extraordinarily difficult.

\begin{example}
A metric space $X,d_X$ is quasi-isometric to a point if and only if it has bounded diameter.
A Cayley graph $C_S(G)$ (for $S$ finite) is quasi-isometric to a point 
if and only if $G$ is finite.
\end{example}

\begin{example}\label{generating_sets_qi}
If $S$ and $T$ are two finite generating sets for a group $G$ then the identity map from $G$
to itself is a quasi-isometry (in fact, a bilipschitz map)
of $G,d_S$ to $G,d_T$. For, there are constants
$C_1$ and $C_2$ so that $d_T(s)\le C_1$ for all $s\in S$, and $d_S(t)\le C_2$ for all
$t\in T$, and therefore $C_2^{-1}d_T(g,h)\le d_S(g,h)\le C_1d_T(g,h)$.
\end{example}

Because of this, the quasi-isometry class of $G,d_S$ is {\em independent} of the
choice of finite generating set, and we can speak unambiguously of the quasi-isometry
class of $G$.

\medskip

The Schwarz Lemma connects the geometry of groups to the geometry of spaces 
they act on.

\begin{lemma}[Schwarz Lemma]\label{Schwarz_lemma}
Let $G$ act properly discontinuously and 
cocompactly by isometries on a proper geodesic metric space $X$. Then 
$G$ is finitely generated by some set $S$, and the
orbit map $G \to X$ sending $g$ to $gx$ (for any $x\in X$) is a quasi-isometry from
$G,d_S$ to $X$.
\end{lemma}
\begin{proof}
Since $X$ is proper and $G$ acts cocompactly there is an $R$ so that $GN_R(x)=X$.
Note that $Gx$ is a net, since every point of $X$ is contained in some translate $gB$
and is therefore within distance $R$ of $gx$. 

Let $B=N_{2R+1}(x)$. Since $G$ acts properly discontinuously, there are only
finitely many $g$ in $G$ for which $gB\cap B$ is nonempty; let $S$ be the nontrivial
elements of this set.

Now, if $g,h \in G$ are arbitrary, 
let $\gamma$ be a geodesic in $X$ from $gx$ to $hx$. Parameterize $\gamma$ by arclength,
and for each integer $i \in (0,|\gamma|)$ let $g_i$ be such that $d_X(g_ix,\gamma(i))\le R$.
Then $g_i^{-1}g_{i+1} \in S$ and therefore 
$$d_S(g,h)=|g^{-1}h| \le |\gamma|+1 = d(gx,hx)+1$$
which shows incidentally that $S$ generates $G$.

Conversely, if $L:=d_S(g,h)$ and
$g_i$ is a sequence of elements with $g_0=g$ and $g_L=h$ and each $g_i^{-1}g_{i+1}\in S$, then
there is a path $\gamma_i$ from $g_ix$ to $g_{i+1}x$ of length at most $4R+2$, and the
concatenation of these paths certifies that 
$$d(gx,hx)\le (4R+2)|g^{-1}h| = (4R+2)d_S(g,h)$$
This completes the proof of the lemma.
\end{proof}

\begin{example}
If $G$ is a group and $H$ is a subgroup of finite index, then $G$ and $H$ are quasi-isometric
(for, both act properly discontinuously and cocompactly on $C_S(G)$).
Two groups are said to be {\em commensurable} if they have isomorphic subgroups of finite
index; the same argument shows that commensurable groups are quasi-isometric.
\end{example}

\begin{example}
Any two regular trees of (finite) valence $\ge 3$ are quasi-isometric; for, any such tree
admits a cocompact action by a free group of finite rank, and any two free groups of
finite rank are commensurable.
\end{example}

\begin{example}
The set of ends of a geodesic metric space is a quasi-isometry invariant.
A famous theorem of Stallings \cite{Stallings_ends} says that a finitely generated
group with more than one end splits over a finite subgroup; it follows that the
property of splitting over a finite subgroup is a quasi-isometry invariant.

Finiteness of the edge groups (in a splitting) is detected quasi-isometrically by the
existence of {\em separating} compact subsets. Quasi-isometry can further detect
the finiteness of the vertex groups, and in particular one observes that a group is
quasi-isometric to a free group if and only if it is virtually free. 
\end{example}

\begin{example}
Any two groups that act cocompactly and properly discontinuously on the same
space $X$ are quasi-isometric. For example, if $M_1,M_2$ are closed Riemannian manifolds with
isometric universal covers, then $\pi_1(M_1)$ and $\pi_1(M_2)$ are quasi-isometric.
It is easy to produce examples for which the groups in question are not commensurable;
for instance, a pair of closed hyperbolic $3$-manifolds $M_1$, $M_2$ with different
invariant trace fields (see \cite{Maclachlan_Reid}). 
\end{example}

\begin{remark}
In the geometric group theory literature, Lemma~\ref{Schwarz_lemma} 
is often called the ``Milnor--\v Svarc (or \v Svarc-Milnor) Lemma''; ``\v Svarc'' 
here is in fact the well-known mathematical physicist Albert Schwarz; it is our 
view that the orthography ``\v Svarc'' tends to obscure this. Actually, the content 
of this Lemma was first observed by Schwarz in the early 50's and only 
rediscovered 15 years later by Milnor at a time when the work of Soviet 
mathematicians was not widely disseminated in the west. 
\end{remark}

\subsection{Hyperbolic spaces}

In a geodesic metric space a {\em geodesic triangle} is just a union of three
geodesics joining three points in pairs. If the three points are $x,y,z$ we
typically denote the (oriented) geodesics by $xy$, $yz$ and $zx$ respectively;
this notation obscures the possibility that the geodesics in question are not
uniquely determined by their endpoints.

\begin{definition}
A geodesic metric space $X,d_X$ is {\em $\delta$-hyperbolic} if for any 
geodesic triangle, each side of the triangle is contained in the
$\delta$-neighborhood of the union of the other two sides.
A metric space is {\em hyperbolic} if it is $\delta$-hyperbolic for some $\delta$.
\end{definition}
One sometimes says that {\em geodesic triangles are $\delta$-thin}. 

\begin{figure}[ht]
\centering
\labellist
\small\hair2pt
\endlabellist
\includegraphics[scale=1]{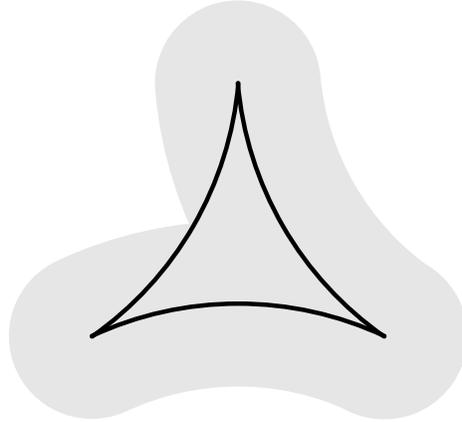}
\caption{A $\delta$-thin triangle; the gray tubes have thickness $\delta$.}
\label{fig:thin_triangle}
\end{figure}
\medskip

\begin{example}
A tree is $0$-hyperbolic.
\end{example}

\begin{example}
Hyperbolic space (of any dimension) is $\delta$-hyperbolic for a uniform $\delta$.
\end{example}

\begin{example}
If $X$ is a simply-connected complete Riemannian manifold with curvature bounded above
by some $K<0$ then $X$ is $\delta$-hyperbolic for some $\delta$ depending on $K$.
\end{example}

\begin{definition}
A geodesic metric space $X$ is $\CAT(K)$ for some $K$ if triangles are
{\em thinner} than comparison triangles in a space of constant curvature $K$. This
means that if $xyz$ is a geodesic triangle in $X$, and $x'y'z'$ is a geodesic triangle in a 
complete simply connected Riemannian manifold $Y$ of constant curvature $K$ with edges of the same
lengths, and $\phi:xyz \to x'y'z'$ is an isometry on each edge, then for any $w \in yz$
we have $d_X(x,w) \le d_Y(x',\phi(w))$.
\end{definition}

The initials $\CAT$ stand for Cartan--Alexandrov--Toponogov, who made substantial 
contributions to the theory of comparison geometry.

\begin{example}
From the definition, a $\CAT(K)$ space is $\delta$-hyperbolic whenever the complete
simply connected Riemannian $2$-manifold of constant curvature $K$ is $\delta$-hyperbolic.
Hence a $\CAT(K)$ space is hyperbolic if $K<0$.
\end{example}

\begin{example}\label{exa:nearest_point}
Nearest point projection to a convex subset of a $\CAT(K)$ space with $K\le 0$ is
distance nonincreasing. Therefore the subspace metric and the path metric on a convex
subset of a $\CAT(K)$ space agree, and such a subspace is itself $\CAT(K)$.
\end{example}

Thinness of triangles implies thinness of arbitrary polygons.

\begin{example}\label{exa:quadrilateral}
Let $X$ be $\delta$-hyperbolic and let $abcd$ be a geodesic quadrilateral. Then either
there are points on $ab$ and $cd$ at distance $\le 2\delta$ or there are points on
$ad$ and $bc$ at distance $\le 2\delta$, or possibly both.
\end{example}

\begin{figure}[ht]
\centering
\labellist
\small\hair2pt
\endlabellist
\includegraphics[scale=1]{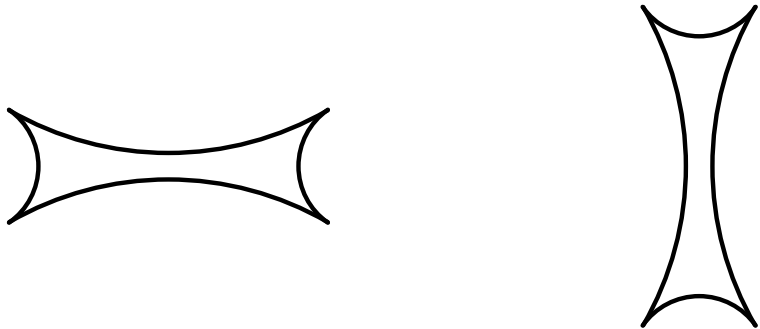}
\caption{Two ways that a quadrilateral can be thin}
\label{fig:thin_quadrilateral}
\end{figure}

The number of essentially distinct ways in which an $n$-gon can be thin is equal
to the $n$th Catalan number. By cutting up a polygon into 
triangles and examining the implications of $\delta$-thinness for each triangle,
one can reason about the geometry of complicated configurations in $\delta$-hyperbolic
space.

\begin{lemma}\label{nearest_point_projection}
Let $X$ be $\delta$-hyperbolic, let $\gamma$ be a geodesic segment/ray/line
in $X$, and let $p \in X$. Then there is a point $q$ on $\gamma$ realizing the infimum of
distance from $p$ to points on $\gamma$, and moreover for any two such points
$q,q'$ we have $d_X(q,q')\le 4\delta$.
\end{lemma}
\begin{proof}
The existence of some point realizing the infimum follows from the properness of
$d(p,\cdot):\gamma \to \R$, valid for any geodesic in any metric space.

Let $q,q'$ be two such points, and
if $d(q,q')>4\delta$ let $q''$ be the midpoint of the segment $qq'$, so
$d(q,q'')=d(q'',q')>2\delta$. Without loss of
generality there is $r$ on $pq$ with $d(r,q'')\le \delta$ hence $d(r,q)>\delta$. But
then $$d(p,q'')\le d(p,r)+d(r,q'') \le d(p,r) + \delta < d(p,r) + d(r,q) = d(p,q)$$
contrary to the fact that $q$ minimizes the distance from $p$ to points on $\gamma$.
\end{proof}

Lemma~\ref{nearest_point_projection} says that there is an approximate 
``nearest point projection'' map $\pi$ from $X$ to any geodesic $\gamma$
(compare with Example~\ref{exa:nearest_point}). This map is not
continuous, but nearby points must map to nearby points, in the sense that
$d(\pi(x),\pi(y))\le d(x,y)+8\delta$.

\medskip

We would now like to show that the property of being hyperbolic is preserved under
quasi-isometry. The problem is that the property of $\delta$-hyperbolicity is
expressed in terms of geodesics, and quasi-isometries do not take geodesics to geodesics.

A {\em quasigeodesic} segment/ray/line is the image of a segment/ray/line
in $\R$ under a quasi-isometric map. For
infinite or semi-infinite intervals this definition has content; for finite intervals
this definition has no content without specifying the constants involved. Hence we can
talk about a $K,\epsilon$ quasigeodesic segment/ray/line.

\begin{lemma}[Morse lemma]\label{morse_lemma}
Let $X,d_X$ be a proper $\delta$-hyperbolic space. Then for any $K,\epsilon$
there is a constant $C$ (depending in an explicit way on $K,\epsilon,\delta$)
so that any $K,\epsilon$ quasigeodesic $\gamma$ is within Hausdorff distance $C$ of
a genuine geodesic $\gamma^g$. If $\gamma$ has one or two endpoints, $\gamma^g$ can be
chosen to have the same endpoints.
\end{lemma}
\begin{proof}
If $\gamma$ is noncompact, it can be approximated on compact subsets by finite segments 
$\gamma_i$. If we prove the lemma for finite segments, then a subsequence of the
$\gamma_i^g$, converging on compact subsets, will limit to $\gamma^g$ with the desired properties
(here is where we use properness of $X$). So it suffices to prove the lemma for
$\gamma$ a segment. 

In this case choose any $\gamma^g$ with the same endpoints as $\gamma$. We need to estimate
the Hausdorff distance from $\gamma$ to $\gamma^g$. Fix some constant $C$
and suppose there are points $p,p'$ on $\gamma$ that are both distance $C$ from
$\gamma^g$, but $d(r,\gamma^g)\ge C$ for all $r$ on $\gamma$ between $p$ and $p'$. Choose
$p_i$ a sequence of points on $\gamma$ and $q_i$ a sequence of
points on $\gamma^g$ closest to the $p_i$ so that $d(q_i,q_{i+1})=11\delta$.

Consider the quadrilateral $p_ip_{i+1}q_{i+1}q_i$. By Example~\ref{exa:quadrilateral}
either there are close points on $p_ip_{i+1}$ and $q_iq_{i+1}$, or close points
on $p_iq_i$ and $p_{i+1}q_{i+1}$ (or possibly both). Suppose there are
points $r_i$ on $p_iq_i$ and $r_{i+1}$ on $p_{i+1}q_{i+1}$ with
$d(r_i,r_{i+1})\le 2\delta$. Then any nearest point projections of $r_i$ and
$r_{i+1}$ to $\gamma^g$ must be at most distance $10\delta$ apart. But
$q_i$ and $q_{i+1}$ are such nearest point projections, by definition, and
satisfy $d(q_i,q_{i+1})=11\delta$. So it must be instead that there
are points $r_i$ on $p_ip_{i+1}$ and $s_i$ on $q_iq_{i+1}$ 
which are at most $2\delta$ apart. But this means that 
$d(p_i,p_{i+1})\ge 2C-4\delta$, so the length of $\gamma$ between $p$ and $p'$ is at
least $(2C-4)d(q,q')/11\delta$ where $q,q'$ are points on $\gamma$ closest to $p,p'$.
On the other hand, $d(p,p')\le 2C+d(q,q')$. Since $\gamma$ is a $K,\epsilon$
quasigeodesic, if $d(q,q')$ is big enough, we get a uniform bound on $C$ 
in terms of $K,\epsilon,\delta$. The remaining case where 
$d(q,q')$ is itself uniformly bounded but $C$ is unbounded quickly 
leads to a contradiction.
\end{proof}

\begin{corollary}\label{delta_hyperbolic_QI}
Let $Y$ be $\delta$-hyperbolic and let $f:X \to Y$ be a $K,\epsilon$ quasi-isometry. Then
$X$ is $\delta'$-hyperbolic for some $\delta$. Hence the property of being hyperbolic
is a quasi-isometry invariant.
\end{corollary}
\begin{proof}
Let $\Gamma$ be a geodesic triangle in $X$ with vertices $a,b,c$. Then the
edges of $f(\Gamma)$ are $K,\epsilon$ quasigeodesics in $Y$, and are therefore within
Hausdorff distance $C$ of geodesics with the same endpoints. It follows that
every point on $f(ab)$ is within distance $2C+\delta$ of $f(ac)\cup f(bc)$
and therefore every point on $ab$ is within distance $K(2C+\delta)+\epsilon$
of $ac\cup bc$.
\end{proof}

The Morse Lemma lets us promote quasigeodesics to (nearby) geodesics. The
next lemma says that quasigeodesity is a {\em local} condition.

\begin{definition}
A path $\gamma$ in $X$ is a {\em $k$-local geodesic} if the subsegments of
length $\le k$ are geodesics. Similarly, $\gamma$ is a {\em $k$-local 
$K,\epsilon$ quasigeodesic} if the subsegments of length $\le k$ are
$K,\epsilon$ quasigeodesics.
\end{definition}

\begin{lemma}[$k$-local geodesics]\label{k_local_geodesics}
Let $X$ be a $\delta$-hyperbolic geodesic space, and let $k>8\delta$. Then
any $k$-local geodesic is $K,\epsilon$ quasigeodesic for $K,\epsilon$
depending explicitly on $\delta$.

More generally, for any $K,\epsilon$ there is a $k$ and constants $K',\epsilon'$
so that any $k$-local $K,\epsilon$ quasigeodesic is a $K',\epsilon'$ quasigeodesic. 
\end{lemma}
\begin{proof}
Let $\gamma$ be a $k$-local geodesic segment from $p$ to $q$, and let
$\gamma^g$ be any geodesic from $p$ to $q$. Let $r$ be a point on $\gamma$
furthest from $\gamma^g$, and let $r$ be the midpoint of an arc $r'r''$
of $\gamma$ of length $8\delta$. By hypothesis, $r'r''$ is actually a geodesic.
Let $s'$ and $s''$ be points on $\gamma^g$ closest to $r'$ and $r''$. The point
$r$ is within distance $2\delta$ either of $\gamma^g$ or of one of the sides
$r's'$ or $r''s''$. If the latter, we would get a path from $r$ to $s'$ or $s''$ 
shorter than the distance from $r'$ or $r''$, contrary to the definition of $r$. 
Hence the distance from $r$ to $\gamma^g$ is at most $2\delta$, and therefore
$\gamma$ is contained in the $2\delta$ neighborhood of $\gamma^g$.

Now let $\pi:\gamma \to \gamma^g$ take points on $\gamma$ to closest points on
$\gamma^g$. Since $\pi$ moves points at most $2\delta$, it is approximately
continuous. Since $\gamma$ is a $k$-local geodesic, the map $\pi$ is approximately
monotone; i.e.\/ if $p_i$ are points on $\gamma$ with $d(p_i,p_{i+1})=k$ moving
monotonely from one end of $\gamma$ to the other, then $d(\pi(p_i),\pi(p_{i+1}))\ge k-4\delta$ 
and the projections also move monotonely
along $\gamma$. In particular, $d(p_i,p_j)\ge (k-4\delta)|i-j|$ and $\pi$ is
a quasi-isometry. The constants involved evidently depend only on $\delta$ and $k$,
and the multiplicative constant evidently goes to $1$ as $k$ gets large.

\medskip

The more general fact is proved similarly, by using Lemma~\ref{morse_lemma} to promote
local quasigeodesics to local geodesics, and then back to global quasigeodesics.
\end{proof}

\subsection{Hyperbolic groups}

Corollary~\ref{delta_hyperbolic_QI} justifies the following definition:

\begin{definition}
A group $G$ is hyperbolic if $C_S(G)$ is $\delta$-hyperbolic
for some $\delta$ for some (and hence for any) finite generating set $S$.
\end{definition}

\begin{example}
Free groups are hyperbolic, since their Cayley graphs (with respect to
a free generating set) are trees which are $0$-hyperbolic.
\end{example}

\begin{example}
Virtually free groups, being precisely the groups quasi-isometric to trees, are
hyperbolic. A group quasi-isometric to a point or to $\R$ is finite or
virtually $\Z$ respectively; such groups are called {\em elementary hyperbolic
groups}; all others are {\em nonelementary}.
\end{example}

\begin{example}
Fundamental groups of closed surfaces with negative Euler characteristic are hyperbolic.
By the uniformization theorem, each such surface can be given a hyperbolic metric,
exhibiting $\pi_1$ as a cocompact group of isometries of the hyperbolic plane.
\end{example}

\begin{example}
A {\em Kleinian group} is a finitely generated discrete subgroup
of the group of isometries of hyperbolic $3$-space. A Kleinian group $G$ is
is {\em convex cocompact} if it acts cocompactly on the convex hull of
its limit set (in the sphere at infinity). Such a convex hull is $\CAT(-1)$, so a convex cocompact
Kleinian group is hyperbolic. See e.g.\/ \cite{Maskit} for an introduction to
Kleinian groups.
\end{example}

\begin{lemma}[invariant quasiaxis]\label{invariant_quasiaxis}
Let $G$ be hyperbolic. Then there are finitely many conjugacy classes of torsion
elements (and therefore a bound on the order of the torsion) and there are
constants $K,\epsilon$ so that for any nontorsion
element $g$ there is a $K,\epsilon$ quasigeodesic $\gamma$ invariant under
$g$ on which $g$ acts as translation.
\end{lemma}
\begin{proof}
Let $g\in G$ be given. Consider the action of $g$ on the Cayley graph $C_S(G)$.
The action is simplicial, so $p \to d(p,gp)$ has no strict local minima in the 
interior of edges, 
and takes integer values at the vertices (which correspond to elements of $G$).
It follows that
there is some $h$ for which $d(h,gh)$ is minimal, and we can take $h$ to be
an element of $G$ (i.e.\/ a vertex).
If $d(h,gh)=k>8\delta$ then we can join $h$ to $gh$
by a geodesic $\sigma$ and let $\gamma=\cup_i g^i\sigma$. Note that $g$ acts
on $\gamma$ by translation through distance $k$; since this is the minimum
distance that $g$ moves points of $G$, it follows that $\gamma$ is a
$k$-local geodesic (and therefore a $K,\epsilon$ quasigeodesic by 
Lemma~\ref{k_local_geodesics}). Note in this case that $g$ has infinite order.

Otherwise there is $h$ moved a least distance by $g$ so 
that $d(h,gh)\le 8\delta$. Since $G$ acts cocompactly on itself, 
there are only finitely many conjugacy classes of elements that move some point
any uniformly bounded distance, so if $g$ is torsion we are done.
If $g$ is not torsion, its orbits are proper, so for any $T$ 
there is an $N$ so that $d(h,g^Nh)>T$; choose $T$ (and $N$) much bigger than some fixed
(but big) $n$. Let $\gamma$ be a geodesic from $h$ to $g^Nh$. Then for any $0\le i\le n$
the geodesic $g^i\gamma$ has endpoints within distance $8\delta n$ of the endpoints of
$\gamma$. On the other hand, $|\gamma|=T\gg8\delta n$ so $\gamma$ contains a segment
$\sigma$ of length at least $T-16\delta n - O(\delta)$ such that $g^i\sigma$ is
contained in the $2\delta$ neighborhood of $\gamma$ for $0\le i \le n$.
To see this, consider the quadrilateral with successive
vertices $h$, $g^Nh$, $g^{i+N}h$ and $g^ih$. Two nonadjacent sides must contain
points which are at most $2\delta$ apart. Since $N\gg i$, the sides must be
$\gamma$ and $g^i\gamma$. We find $\sigma$ and $g^i\sigma$ in the region where these
two geodesics are close.

Consequently, for any $p \in \sigma$ the sequence $p,gp,\cdots, g^np$
is a $K,\epsilon$ quasigeodesic for some uniform $K,\epsilon$ {\em independent of $n$}.
In particular there is a constant $C$ (independent of $n$)
so that $d(p,g^ip)\ge iC$ for $0\le i\le n$, and therefore
the infinite sequence $g^ip$ for $i\in \Z$
is an $(nC)$-local $K,\epsilon$ quasigeodesic. Since $K,\epsilon$ is fixed, if $n$ is big
enough, this infinite sequence is an honest $K',\epsilon'$ quasigeodesic invariant
under $g$, by Lemma~\ref{k_local_geodesics}.
Here $K',\epsilon'$ depends only on $\delta$ and $G$, and not on $g$.
\end{proof}

Lemma~\ref{invariant_quasiaxis} can be weakened considerably, and it is frequently
important to study actions which are not necessarily cocompact 
on $\delta$-hyperbolic spaces which are not necessarily proper.
The quasigeodesic $\gamma$ invariant under $g$ is called a {\em quasiaxis}. 
Quasiaxes in $\delta$-hyperbolic spaces are (approximately) {\em unique}:

\begin{lemma}
Let $G$ be hyperbolic, and let $g$ have infinite order. Let $\gamma$ and
$\gamma'$ be $g$-invariant $K,\epsilon$ quasigeodesics (i.e.\/ quasiaxes for $g$).
Then $\gamma$ and $\gamma'$ are a finite Hausdorff distance apart, and this
finite distance depends only on $K,\epsilon$ and $\delta$.
Consequently the centralizer $C(g)$ is virtually $\Z$.
\end{lemma}
\begin{proof}
Let $p\in\gamma$ and $p'\in\gamma'$ a closest point to $p$. 
Since $g$ acts on both $\gamma$ and $\gamma'$
cocompactly, there is a constant $C$ so that every point in $\gamma$ or $\gamma'$
is within $C$ from some point in the orbit of $p$ or $p'$. This implies that
the Hausdorff distance from $\gamma$ to $\gamma'$ is at most $2C+d(p,p')$; in particular,
this distance is finite.

Pick two points on $\gamma$ 
very far away from each other; each is distance at most $2C+d(p,p')$ from $\gamma'$, 
and therefore most of the geodesic between them is within
distance $2\delta$ of the geodesic between corresponding points on $\gamma'$. But
$\gamma$ and $\gamma'$ are themselves $K,\epsilon$ quasigeodesic, and therefore uniformly
close to these geodesics. Hence some points on $\gamma$ are within a uniformly
bounded distance of $\gamma'$, and therefore all points on $\gamma$ are.

If $h$ commutes with $g$, then $h$ must permute the quasiaxes of $g$. Therefore
$h$ takes points on any quasiaxis $\gamma$ for $g$
to within a bounded distance of $\gamma$. Hence $C(g)$, thought of as a subset of $G$, is
quasiisometric to a quasiaxis (that is to say, to $\R$), and is therefore virtually $\Z$. 
\end{proof}

This shows that a hyperbolic group cannot contain a copy of $\Z\oplus\Z$ 
(or, for that matter, the fundamental group of a Klein bottle). 
This is more subtle than it might seem;
$\Z\oplus\Z$ {\em can} act freely and properly discontinuously by isometries on a proper
$\delta$-hyperbolic space --- for example, as a parabolic subgroup of the isometries
of $\H^3$.

\begin{example}
If $M$ is a closed $3$-manifold, then $\pi_1(M)$ is hyperbolic if and only if it does not
contain any $\Z \oplus \Z$ subgroup. Note that this includes the possibility that
$\pi_1(M)$ is elementary hyperbolic (for instance, finite).
This follows from Perelman's Geometrization Theorem
\cite{Perelman_1,Perelman_2}.
\end{example}

If $g$ is an isometry of any metric space $X$, the {\em translation length} of $g$ is
the limit $\tau(g):=\lim_{n \to \infty} d_X(p,g^np)/n$ for some $p\in X$. The triangle
inequality implies that the limit exists and is independent of the choice of $p$. Moreover,
from the definition, $\tau(g^n)=|n|\tau(g)$ and $\tau(g)$ is a conjugacy invariant.

Lemma~\ref{invariant_quasiaxis} implies that for $G$ acting on itself, $\tau(g)=0$ if and
only if $g$ has finite (and therefore bounded) order. Consequently a hyperbolic
group cannot contain a copy of a Baumslag--Solitar group; i.e.\/ a group
of the form $BS(p,q):=\langle a,b \; | \; ba^pb^{-1}=a^q\rangle$. For, we have
already shown hyperbolic groups do not contain $\Z\oplus\Z$, and this rules out the case
$|p|=|q|$, and if $|p|\ne|q|$ then for any isometric action of $BS(p,q)$ 
on a metric space, $\tau(a)=0$.

By properness of $C_S(G)$ and the Morse Lemma, there is a constant $N$ so that
for any $g\in G$ the power $g^N$ has an invariant {\em geodesic} axis on which it acts
by translation. It follows that $\tau(g)\in \Q$, and in fact $\in \frac 1 N\Z$;
this cute observation is due to Gromov \cite{Gromov_hyperbolic}.

\subsection{The Gromov boundary}

Two geodesic rays $\gamma,\gamma'$ in a metric space $X$ are {\em asymptotic} if they
are a finite Hausdorff distance apart. The property of being asymptotic is an equivalence
relation, and the set of equivalence classes is the {\em Gromov boundary}, and denoted
$\partial_\infty X$. If $X$ is proper and $\delta$-hyperbolic,
and $x$ is any basepoint, then
every equivalence class contains a ray starting at $x$. For, if $\gamma$ is a
geodesic ray, and $g_i \in \gamma$ goes to infinity, then by properness,
any collection of geodesics $xg_i$ contains a subsequence which 
converges on compact subsets to a ray $\gamma'$. By $\delta$-thinness each of the triangles
$xg_0g_i$ is contained in a uniformly bounded neighborhood of $\gamma$, so the same
is true of $\gamma'$; in particular, $\gamma'$ 
is asymptotic to $\gamma$. We give $\partial_\infty X$ the topology of convergence
on compact subsets of equivalence classes. That is, $\gamma_i \to \gamma$ if and only if 
every subsequence of the $\gamma_i$ contains a further subsequence 
whose equivalence classes have representatives that converge 
on compact subsets to some representative of the equivalence class of $\gamma$.

\begin{lemma}\label{boundary_compact}
Let $X$ be a $\delta$-hyperbolic proper geodesic metric space. Then $\partial_\infty X$ is
compact.
\end{lemma}
\begin{proof}
If $\gamma_i$ is any sequence of rays, and $\gamma_i'$ is an equivalent sequence starting
at a basepoint $x$, then by properness $\gamma_i'$ has a subsequence which 
converges on compact subsets.
\end{proof}
In fact, we can define a (compact) topology on $\overline{X}:=X\cup \partial_\infty X$
by saying that $x_i \to \gamma$ if and only if every subsequence of a sequence of
geodesics $xx_i$ contains a further subsequence which converges on compact subsets to a
representative of $\gamma$. With this topology, $\overline{X}$ is compact, 
$\partial_\infty X$ is closed in $\overline{X}$, and the inclusion
of $X$ into $\overline{X}$ is a homeomorphism onto its image.

\medskip

A bi-infinite geodesic $\gamma$ determines two (distinct) points in $\partial_\infty X$;
we call these the {\em endpoints} of $\gamma$. Two geodesics with the same
(finite or infinite) endpoints are Hausdorff distance at most $2\delta$ apart.
Conversely, any two distinct points in
$\partial_\infty X$ are spanned by an infinite geodesic $\gamma$. For, if $\gamma_1,\gamma_2$
are two infinite rays (starting at a basepoint $x$ for concreteness), and $g_i,h_i$
are points on $\gamma_1,\gamma_2$ respectively going to infinity, some point $p_i$ on any
geodesic $g_ih_i$ is within $\delta$ of both $xg_i$ and $xh_i$, and if $p_i \to \infty$
then $\gamma_1$ and $\gamma_2$ would be a finite Hausdorff distance apart. Otherwise some
subsequence of the $p_i$ converges to $p$, and the geodesics $g_ih_i$ converge on compact
subsets to a (nonempty!) bi-infinite geodesic $\gamma$ through $p$ asymptotic to both
$\gamma_1$ and $\gamma_2$. Evidently, geodesic triangles with some or all endpoints at infinity
are $\delta'$-thin for some $\delta'$ depending only on $\delta$ (one can take $\delta'=20\delta$).
By abuse of notation, in the sequel we will call a metric space $\delta$-hyperbolic 
if all geodesic triangles --- even those with some endpoints at infinity --- are $\delta$-thin.

Let $X,Y$ be hyperbolic geodesic metric spaces. Then any
quasi-isometric map $\phi:X \to Y$ extends uniquely to a 
{\em continuous} map $\partial_\infty X \to \partial_\infty Y$. 
In particular, the Gromov boundary $\partial_\infty X$ depends (up to homeomorphism)
only on the quasi-isometry type of $X$, and $\QI(X)$ acts
on $\partial_\infty X$ by homeomorphisms.

If $G$ is a hyperbolic group, we define $\partial_\infty G$ to be the Gromov boundary of
some (any) $C_S(G)$. 

\begin{example}
\begin{figure}[ht]
\centering
\labellist
\small\hair2pt
\endlabellist
\includegraphics[scale=1]{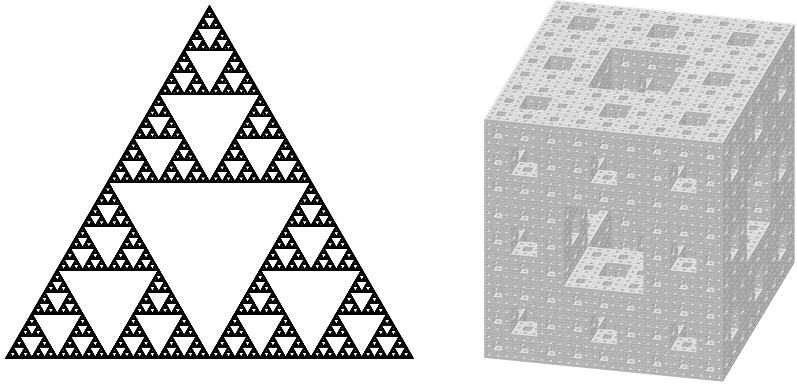}
\caption{The Sierpinski carpet and the Menger sponge.}
\label{fig:carpet_and_sponge}
\end{figure}
If $G$ is free, $\partial_\infty G$ is a Cantor set. If $G$ is a $\pi_1$ of a closed surface with
negative Euler characteristic, $\partial_\infty G$ is a circle. If $G$ is a convex cocompact 
Kleinian group, $\partial_\infty G$ is homeomorphic to the limit set.
For example, if $G$ is the 
fundamental group of a hyperbolic $3$-manifold with totally geodesic boundary,
$\partial_\infty G$ is a Sierpinski carpet.

In fact, a theorem of Kapovich--Kleiner \cite{Kapovich_Kleiner} says that if $G$
is a hyperbolic group which does not split over a finite or virtually cyclic subgroup,
and if $\partial_\infty G$ is $1$-dimensional (in the topological sense of dimension),
then $\partial_\infty G$ is homeomorphic to the circle, the Sierpinski carpet, 
or the Menger sponge.
\end{example}

Evidently, $\partial_\infty G$ is empty if and only if $G$ is finite, and
if $\partial_\infty G$ is nonempty, it has at least two points, and has exactly two points
if and only if $G$ is itself quasi-isometric to the geodesic joining these two points, which 
holds if and only if $G$ is virtually $\Z$.

If $g\in G$ has infinite order, a quasiaxis $\gamma$ is asymptotic
to two points $p^\pm \in \partial_\infty G$. Under (positive) powers of $g$, points stay 
a constant distance from $\gamma$, and move towards one of the endpoints, say $p^+$.
As homeomorphisms from $\overline{X}$ to itself, the elements $g^n$ with $n\to \infty$
converge uniformly (in the compact-open topology) on $\overline{X}-p^-$ to the constant 
map to $p^+$. We call $p^+$ the {\em attracting} endpoint and $p^-$ the {\em repelling}
endpoint of $g$; the actions of $g$ on $\partial_\infty G$ is sometimes expressed by saying
that it has {\em source-sink dynamics}.

\begin{example}\label{free_from_pingpong}
Let $g,h \in G$ be of infinite order, with quasiaxes $\gamma$ and $\gamma'$. If $\gamma$ and
$\gamma'$ share an endpoint (without loss of generality the attracting endpoint of each) and
$p$ is close to both $\gamma$ and $\gamma'$, then there are $n_i,m_i \to \infty$ for which
$d(h^{-m_i}g^{n_i}p,p)$ is bounded. Since the action of $G$ on its Cayley graph is properly
discontinuously, it follows that there are distinct $i,j$ with $h^{-m_i}g^{n_i}=h^{-m_j}g^{n_j}$
so that $h^m=g^n$ for some positive $n,m$. In particular,
in this case $g$ and $h$ together generate a virtual $\Z$ subgroup, and their quasiaxes have the
same endpoints. Otherwise the endpoints are disjoint, and because of the source-sink dynamics, 
Klein's pingpong argument implies that sufficiently large powers $g^n,h^m$ generate a 
(nonabelian) {\em free} subgroup of $G$.
\end{example}

\begin{lemma}\label{action_minimal}
Suppose $G$ is nonelementary. Then the action of $G$ on $\partial_\infty G$ is minimal;
i.e.\/ every orbit is dense. Consequently $\partial_\infty G$ is infinite and perfect.
\end{lemma}
\begin{proof}
If $G$ is nonelementary, there are $g,h$ whose quasiaxes have distinct endpoints $p^\pm$
and $q^\pm$ respectively. If $r\in\partial_\infty G$ is arbitrary, then either
$g^nr \to p^+$ or $g^nhr \to p^+$; it follows that every attracting/repelling point is
in the closure of every orbit.

Now let $\gamma$ be a geodesic from $p^-$ to $p^+$ and let $\gamma'$ be a geodesic ray 
asymptotic to $r$. Pick $s$ on $\gamma$ and let $g_i$ be a sequence of elements with
$g_i(s) \in \gamma'$ converging to $r$. At most one component of $\gamma - s$ can come close
to the basepoint $x$. Hence there is some subsequence so that either $g_ip^+ \to r$ or
$g_ip^- \to r$, and therefore every point is in the closure of the orbit of some
attracting/repelling point. This proves the lemma.
\end{proof}

Another way to see the compactification $\partial_\infty X$ is in terms of 
(equivalence classes of) {\em horofunctions}.

\begin{definition}[horofunction]
Let $\gamma$ be a geodesic ray parameterized by length. The {\em horofunction}
(also called the {\em Busemann function}) associated to $\gamma$ is the limit
$$b_\gamma(x):=\lim_{t \to \infty} d_X(x,\gamma(t))-t$$
The level sets of horofunctions are called {\em horospheres}.
\end{definition}
This limit exists and is finite, by the triangle inequality. Moreover, it is $1$-Lipschitz.
If $\gamma$ and $\gamma'$ are asymptotic, then there is some constant $C(\gamma,\gamma')$
so that $|(b_\gamma - b_{\gamma'}) - C|\le 2\delta$. If $x$ is the endpoint of $\gamma$, we let
$b_x$ denote any horofunction of the form $b_\gamma$, and say that $b_x$ is {\em centered at
$x$}. 

Here is another way to define $b_\gamma$
without reference to $\gamma$. On any proper metric space, the set of $1$-Lipschitz functions 
mod constants is compact (in the topology of convergence on compact subsets). For any 
$x\in X$ the function $d_X(x,\cdot):X \to \R$ is $1$-Lipschitz, and $x \to d_X(x,\cdot)$
embeds $X$ in the space of $1$-Lipschitz functions on $X$ mod constants. 
The closure of this image defines a
natural compactification of $X$; quotienting further by bounded functions gives $\overline{X}$.
For each $x \in \partial_\infty X$ the preimage is the set of equivalence classes of
functions $b_x$. In this way we think of $b_x$ as a normalization of the function which
measures ``distance to $x$''.

\medskip

The space $\partial_\infty X$ can be metrized following Gromov (see \cite{Gromov_hyperbolic}).
\begin{definition}\label{a_distance}
Fix some basepoint $x$ and some constant $a>1$. The {\em $a$-length} of a rectifiable
path $\gamma$ in $X$ is the integral along $\gamma$ of $a^{-d_X(x,\cdot)}$, and the
{\em $a$-distance} from $y$ to $z$, denoted $d_X^a(y,z)$, 
is the infimum of the $a$-lengths of paths between $y$ and $z$.
\end{definition}
A straightforward calculation shows that there is an $a_0>1$ so that for $a<a_0$,
the $a$-length defines a metric on $\overline{X}$. In fact, any $a_0$ with 
$\delta\log(a_0)\ll 1$ will work. If $a$ is too big, $a$-length still extends to
a pseudo-metric on $\overline{X}$, but now distinct points of $\partial_\infty X$
might be joined by a sequence of paths with $a$-length going to $0$.
Increasing $a$ will decrease the Hausdorff dimension of 
$\partial_\infty X$; of course, the Hausdorff dimension must always be at least 
as big as the topological dimension. In any case, it follows that $\partial_\infty X$ is
metrizable.

The following lemma is useful to compare length and $a$-length.
\begin{lemma}\label{a_comparison_lemma}
For $a<a_0$ there is a constant $\lambda$ so that for all points 
$y,z\in \partial_\infty X$ there is an inequality 
$\lambda^{-1}a^{-d_X(x,yz)} \le d_X^a(y,z) \le \lambda a^{-d_X(x,yz)}$ where 
$d_X(x,yz)$ is the ordinary distance from the basepoint $x$ to the geodesic $yz$.
\end{lemma}
For a proof, see \cite{Coornaert_Delzant_Papadopoulos}.

The quantity $d_X(x,yz)$ is sometimes abbreviated by $(y|z)$ (the basepoint $x$ is
suppressed in this notation), and called the {\em Gromov product}. So we can also write
$\lambda^{-1}a^{-(y|z)} \le d_X^a(y,z) \le \lambda a^{-(y|z)}$.
Because of this inequality, different choices of $a$ give rise to H\"older equivalent
metrics on $\partial_\infty X$.
If $X$ is a group $G$, we take $\id$ as the basepoint, by convention.

\begin{remark}
With our notation, $(y|z):=d_X(x,yz)$ is ambiguous, since it depends on a choice of geodesic
from $y$ to $z$. Since we only care about $(y|z)$ up to a uniform additive constant, we ignore
this issue. One common normalization, adopted by Gromov, is to use the formula 
$(y|z): = \frac 1 2( d_X(x,y) + d_X(x,z) - d_X(y,z) )$. These definitions are interchangeable
for our purposes, as the ambiguity can always be absorbed into some unspecified constant.
\end{remark}

A group $G$ acting by homeomorphisms on a compact metrizable space $M$ is said to be
a {\em convergence action} if the induced
action on the space $M^3-\Delta$ of distinct ordered triples is properly discontinuous.

\begin{lemma}\label{action_convergence}
The action of $G$ on $\partial_\infty G$ is a convergence action. Moreover, the action
on the space of distinct triples is cocompact.
\end{lemma}
\begin{proof}
If $x,y,z$ is a distinct triple of points in $\partial_\infty G$, there is a 
point $p$ within distance $\delta$ of all three geodesics $xy,yz,zx$; moreover, the set of
such points has uniformly bounded diameter in $G$. This defines an approximate map from
distinct triples to points in $G$. Since the action of $G$ on itself is cocompact, the same
is true for the action on the space of distinct triples. Similarly, if the action of $G$
on the space of distinct triples were not properly discontinuous, we could find two bounded
regions in $G$ and infinitely many $g_i$ in $G$ taking some point in one bounded region to some
point in the other, which is absurd.
\end{proof}

The converse is a famous theorem of Bowditch:

\begin{theorem}[Bowditch's convergence theorem \cite{Bowditch} Thm.~0.1]\label{bowditch_theorem}
Let $G$ act faithfully, properly discontinuously and cocompactly on the space of 
distinct triples of some perfect compact metrizable space $M$. Then $G$ is hyperbolic 
and $M$ is $G$-equivariantly homeomorphic to $\partial_\infty G$.
\end{theorem}

\subsection{Patterson--Sullivan measure}\label{Patterson_Sullivan_subsection}

The results in this section are due to Coornaert \cite{Coornaert}, although because of
our more narrow focus we are able to give somewhat different and shorter proofs.
However by and large our proofs, like Coornaert's, are obtained by directly generalizing 
ideas of Sullivan \cite{Sullivan} in the context of Kleinian groups.

Let $G$ be a hyperbolic group, and let $G_{\le n}$ denote the set of elements of
(word) length $\le n$, with respect to some fixed generating set. The {\em critical exponent}
$h(G)$ (also called the {\em volume entropy} of $G$) is the quantity
$$h(G):=\limsup_{n\to \infty} \frac 1 n \log{|G_{\le n}|}$$
in other words, the exponential growth rate of $G$.
Since every nonelementary hyperbolic group contains many free groups 
(Example~\ref{free_from_pingpong}), $h(G)=0$ if and only if $G$ is elementary.

Define the (Poincar\'e) {\em zeta function} by the formula
$$\zeta_G(s):= \sum_{g \in G} e^{-s|g|}$$
Then $\zeta_G(s)$ diverges if $s<h(G)$ and converges if $s>h(G)$.

\begin{lemma}\label{zeta_diverges}
The zeta function diverges at $s=h(G)$.
\end{lemma}
\begin{proof}
We will show in \S~\ref{combing_section} (Theorem~\ref{Cannon_theorem})
that for any hyperbolic group $G$ and any
generating set $S$ there is a {\em regular language}
$L\subset S^*$ consisting of geodesics, which evaluates bijectively to $G$. 
In particular, $|G_{\le n}| = |L_{\le n}|$ for any $n$.
In any regular language $L$ the generating function $\sum |L_{\le n}|t^n$ is
{\em rational} (Theorem~\ref{regular_language_generating_function}); 
i.e.\/ it is the power series expansion of $p(t)/q(t)$ for some integral
polynomials $p,q$, and consequently $C^{-1}(e^{hn}n^k)\le|L_{\le n}|\le C(e^{hn}n^k)$ 
for some real $h$ and non-negative integer $k$, and constant $C$. 
Evidently, for $L$ as above, $h=h(G)$ and the zeta function diverges at $h$.
\end{proof}

For $s>h(G)$ construct a probability measure $\nu_s$ on $\overline{G}$ 
(i.e.\/ on $G\cup \partial_\infty G$) supported in $G$, by putting a Dirac 
mass of size $e^{-s|g|}/\zeta_G(s)$ at each $g\in G$. As $s$ converges to $h$ 
from above, this sequence of probability measures contains a subsequence which converges
to a limit $\nu$. Since the zeta function diverges at $h$, the limit $\nu$ is
supported on $\partial G$. This measure is called a {\em Patterson--Sullivan} measure,
by analogy with the work of Patterson and Sullivan \cite{Patterson,Sullivan} on Kleinian groups.

For any $g$, the pushforward of measure $g_*\nu_s$ is defined
by $g_*\nu_s(A) = \nu_s(g^{-1}A)$, and similarly for $g_*\nu$.
For any $g,g'$ there is an inequality $|g'|-|g| \le |gg'| \le |g'|+|g|$. From
the definition on $\nu_s$, this implies that $g_*\nu_s$ is absolutely continuous
with respect to $\nu_s$, and its Radon--Nikodym derivative satisfies
$e^{-s|g|} \le d(g_*\nu_s)/d\nu_s \le e^{s|g|}$.
Passing to a limit we deduce that $e^{-h|g|} \le d(g_*\nu)/d\nu \le e^{h|g|}$. 

The most important property of the measure $\nu$ is a refinement of this inequality, which
can be expressed by saying that it is a so-called {\em quasiconformal
measure of dimension $h$}. The ``conformal'' structure on $\partial_\infty X$ is defined
using the $a$-distance for some fixed $a>1$ (recall Definition~\ref{a_distance}).

\begin{definition}[Coornaert]
For $g\in G$ define $j_g:\partial_\infty X \to \R$ by 
$$j_g(y)=a^{b_y(\id) - b_y(g)}$$
for some horofunction $b_y$ centered at $y$.
A probability measure $\nu$ on $\partial_\infty X$ is a {\em quasiconformal measure of dimension $D$}
if $g_*\nu$ is absolutely continuous with respect to $\nu$ for every $g\in G$, and there
is some constant $C$ independent of $g$ so that
$$C^{-1}j_g(y)^D \le d(g_*\nu)/d\nu \le Cj_g(y)^D$$
\end{definition}
Notice that the ambiguity in the choice of horofunction $b_y$ is absorbed into the
definition of $j_g$ (which only depends on $b_y$ mod constant functions) 
and the constant $C$. The support of any quasiconformal measure is evidently
closed and $G$-invariant, so by Lemma~\ref{action_minimal}, it is all of $\partial_\infty G$.

From the definition of the Radon--Nikodym derivative,
$\nu$ is a quasiconformal measure of dimension $D$ if there is a constant $C$
so that for all $y$ we can find a neighborhood $V$ of $y$ in $\overline{X}$ for which
$$C^{-1}j_g(y)^D\nu(A) \le \nu(g^{-1}A) \le Cj_g(y)^D\nu(A)$$
for all $A\subset V$.

\begin{remark}
For some reason, Coornaert chooses to work with {\em pullbacks} 
of measure $g^*\nu:=g^{-1}_*\nu$ instead of pushforward. Therefore the roles
of $g$ and $g^{-1}$ are generally interchanged between our discussion and
Coornaert's.
\end{remark}

\begin{theorem}[Coornaert \cite{Coornaert} Thm.~5.4]\label{measure_is_conformal}
The measure $\nu$ is a quasiconformal measure of dimension $D$ where $D=h/\log{a}$.
\end{theorem}
\begin{proof}
Evidently the support of $\nu$ is $G$-invariant, and is therefore equal to all of
$\partial_\infty G$. Let $y\in \partial_\infty X$, let $b_y$ be a horofunction centered
at $y$, and let $g\in G$. 

By $\delta$-thinness and the definition of a horofunction,
$ d(g,z)-d(\id,z)$ is close to $b_y(g) - b_y(\id)$ for $z$ sufficiently close to $y$.
In particular, there is a neighborhood $V$ of $y$ in $\overline{X}$ so that
$$|g^{-1}z| - |z| - C \le b_y(g) - b_y(\id) \le |g^{-1}z| - |z| + C$$
for some $C$, and for all $z$ in $V$.

For each $s>h$ we have $g_*\nu_s(z)/\nu_s(z) = \nu_s(g^{-1}z)/\nu_s(z) = e^{-s(|g^{-1}z|-|z|)}$.
Taking the limit as $s\to h$ and defining $D$ by $a^D=e^h$ proves the theorem.
\end{proof}

To make use of this observation, we introduce the idea of a {\em shadow},
following Sullivan.

\begin{definition}
For $g\in G$ and $R$ a positive real number, the {\em shadow} $S(g,R)$ is the
set of $y\in \partial_\infty G$ such that every geodesic ray from $\id$ to $y$
comes within distance $R$ of $g$.
\end{definition}

Said another way, $y$ is in $S(g,R)$ if $g$ comes within distance $R$
of any geodesic from $\id$ to $y$. Given $R>2\delta$, 
for any fixed $n$ the shadows $S(g,R)$ with $|g|=n$ cover $\partial_\infty G$
efficiently:

\begin{lemma}\label{uniform_covering_by_shadows}
Fix $R$. Then there is a constant $N$ so that for any $y\in \partial_\infty G$ and
any $n$ there is at least $1$ and there are at most $N$ elements $g$ with
$|g|=n$ and $y \in S(g,R)$.
\end{lemma}
\begin{proof}
If $R>2\delta$, if $\gamma$ is any geodesic from $\id$ to $y$, and if $g$ is any
point on $\gamma$, then $y\in S(g,R)$. Conversely, if $g$ and $h$ are any two
elements with $|g|=|h|$ and $y\in S(g,R)\cap S(h,R)$ then $d(g,h)\le 2R$.
\end{proof}

Sullivan's fundamental observation is that the action of $g^{-1}$ on $S(g,R)$ is
uniformly close to being {\em linear}, in the sense that the derivative
$d(g_*\nu)/d\nu$ varies by a bounded multiplicative constant on $S(g,R)$:

\begin{lemma}\label{linear_on_shadow}
Fix $R$. Then there is a constant $C$ so that for any $y\in S(g,R)$ 
there is an inequality $$C^{-1}a^{|g|} \le j_g(y) \le Ca^{|g|}$$
\end{lemma}
\begin{proof}
Recall $j_g(y) = a^{b_y(\id) - b_y(g)}$ for some horofunction $b_y$.
But by $\delta$-thinness and the definition of a shadow, there is a constant $C'$
so that 
$$|g|-C' \le b_y(\id) - b_y(g) \le |g|+C'$$
for any $y$ in $S(g,R)$.
\end{proof}

From this one readily obtains a uniform estimate on the measure of a shadow:

\begin{lemma}\label{size_of_shadow}
Fix $R$. Then there is a constant $C$ so that for any $g \in G$ there is an inequality
$$C^{-1} a^{-|g|D} \le \nu(S(g,R)) \le C a^{-|g|D}$$
\end{lemma}
\begin{proof}
Let $m_0<1$ be the measure of the biggest atom of $\nu$, and fix $m_0<m<1$. By compactness
of $\partial_\infty G$ there is some $\epsilon$ so that every ball in $\partial_\infty G$
of diameter $\le \epsilon$ has mass at most $m$. Now,
$g^{-1}S(g,R)$ is the set of $y \in \partial_\infty G$ for which every
geodesic ray from $g^{-1}$ to $y$ comes within distance $R$ of $\id$. As $R\to \infty$, the
diameter of $\partial_\infty G - g^{-1}S(g,R)$ goes to zero uniformly in $g$ (this
follows from the quasi-equivalence of $d_X^a(y,z)$ and $a^{-(y|z)}$; see
Lemma~\ref{a_comparison_lemma}). Consequently there is some $R_0$ so that
for all $R\ge R_0$ the measure $\nu(g^{-1}S(g,R))$ is between $1-m$ and $1$, independent of $g$.

Now, by Lemma~\ref{linear_on_shadow} and the definition of a quasiconformal measure,
there is a constant $C_1$ so that
$$C_1^{-1} a^{|g|D} \le \nu(g^{-1}S(g,R))/\nu(S(g,R)) \le C_1 a^{|g|D}$$
Taking reciprocals, and using the fact that $1-m \le \nu(g^{-1}S(g,R)) \le 1$
completes the proof (at the cost of adjusting constants).
\end{proof}

Note that the argument shows that $\nu$ has no atoms, since any $y\in \partial_\infty G$
is contained in some shadow of measure $\le C a^{-Dn}$ for any $n$.
We deduce the following corollary.

\begin{corollary}[Coornaert \cite{Coornaert} Thm.~7.2]\label{Coornaert_growth_rate}
Let $G$ be a hyperbolic group. Then there is a constant $C$ so that
$$C^{-1}e^{hn} \le |G_{\le n}| \le Ce^{hn}$$ for all $n$.
\end{corollary}
\begin{proof}
The lower bound is proved in \S~\ref{combing_section}, so we just need to prove the upper
bound. 

For each $g$ with $|g|=n$ Lemma~\ref{size_of_shadow} says
$e^{-hn} = a^{-Dn} \le C_1\nu(S(g,R))$. On the other hand, Lemma~\ref{uniform_covering_by_shadows}
says that every point $y \in\partial_\infty G$
is in at most $N$ sets $S(g,R)$ with $|g|=n$, so
$$|G_n|e^{-hn}C_1^{-1} \le \sum_{|g|=n} \nu(S(g,R)) \le N\nu(\cup_{|g|=n} S(g,R))=N$$
\end{proof}

Corollary~\ref{Coornaert_growth_rate} has important consequences that we will explore in
\S~\ref{combing_section}.

\medskip

A second corollary gives very precise metric and dynamical control 
over $\partial_\infty G$. An action of a group $G$ on a space $X$ is said to be
{\em ergodic} for some measure $\nu$ on $X$ if for any two subsets $A$, $B$
of $X$ with $\nu(A),\nu(B)>0$ there is some $g\in G$ with $\nu(g(A)\cap B)>0$.

\begin{corollary}[Coornaert \cite{Coornaert} Cor.~7.5 and Thm.~7.7]\label{ergodic_at_infinity}
Let $\nu$ be a quasiconformal measure on $\partial_\infty G$ of dimension $D$. Then
$\nu$ is quasi-equivalent to $D$-dimensional Hausdorff measure; i.e.\/ there is a constant $C$
so that $C^{-1}\Haus^D(A) \le \nu(A) \le C\Haus^D(A)$ for any $A$. In particular,
the space $\partial_\infty G$ has Hausdorff dimension $D$, and its $D$-dimensional Hausdorff
measure is finite and positive. Moreover, the action of $G$ on $\partial G$ 
is ergodic for $\nu$.
\end{corollary}
\begin{proof}
Evidently, the second and third claims follow from the first (if $A$ is a $G$-invariant subset
of $\partial_\infty G$ of positive $\nu$-measure, 
the restriction of $\nu$ to $A$ is a quasiconformal measure of dimension $D$, and is therefore
quasi-equivalent to $\Haus^D$ and thence to $\nu$. In particular, $A$ has full measure). So it
suffices to show that $\nu$ and $\Haus^D$ are quasi-equivalent.

Since $C_1^{-1}a^{-(y|z)} \le d_X^a(y,z)\le C_1a^{-(y|z)}$ it follows that
every metric ball $B(y,r)$ in $\partial_\infty G$ can be sandwiched between
two shadows $S(g_1,R) \subset B(y,r) \subset S(g_2,R)$ with $a^{-|g_1|} \ge r/C_2$ and $a^{-|g_2|}\le rC_2$. 
From Lemma~\ref{size_of_shadow} we obtain
$C_2^{-1}r^D \le \nu(B(y,r)) \le C_2r^D$.
From this and the definition of Hausdorff measure, we will obtain the theorem.

If $A$ is any measurable set, cover $A$ by balls $U_i$ of radius $\epsilon_i\le \epsilon$. Then 
$$\nu(A) \le \nu(\cup_i U_i) \le \sum_i \nu(U_i) \le C_2 \sum \epsilon_i^D$$
so letting $\epsilon \to 0$ we get $\nu(A) \le C_2\Haus^D(A)$.

The following proof of the reverse inequality was suggested to us by Curt McMullen. 
For any $\delta$ let $K$ be compact and $U$ open so that $K \subset A \subset U$
and both $\nu(U-K)$ and $\Haus^D(U-K)$ are less than $\delta$. By compactness, there is
an $\epsilon$ so that every ball of radius $\le \epsilon$ centered at a point in $K$ is
contained in $U$. Now inductively cover $K$ by balls $U_1,U_2,\cdots$
of non-increasing radius $\epsilon_i\le \epsilon$ in such a way that the center of each $U_i$ is not in
$U_j$ for and $j<i$. Then the balls with the same centers and half the radii are disjoint, so
$$\sum \epsilon_i^D = 2^D \sum (\epsilon_i/2)^D \le C_3\nu(U)$$
and therefore $\Haus^D(K) \le C_3\nu(U)$. Taking $\delta\to 0$ gives $\Haus^D(A)\le C_3\nu(A)$
and we are done.
\end{proof}

\begin{remark}
Coornaert only gives the proof of the inequality $\nu(A) \le CH^D(A)$ in his paper,
referring the reader to Sullivan \cite{Sullivan} for the proof of $C^{-1}H^D(A) \le \nu(A)$.
However, there is a gap in Sullivan's proof of the reverse inequality, of which the
reader should be warned.
\end{remark}

\begin{remark}
The approximate linearity of $g^{-1}$ on $S(g,R)$ has many other applications.
For example, see the proof of Theorem 1 in \cite{Sullivan_ergodic}.
\end{remark}

\section{Combings}\label{combing_section}

On a Riemannian manifold, a ``geodesic'' is just a smooth path that locally minimizes
length (really, energy). A sufficiently long geodesic is typically not globally length
minimizing, and the entire subject of Morse theory is devoted to the
difference. By contrast, one of the most important qualitative features of negative curvature 
is that (quasi)-geodesity is a {\em local} property (i.e.\/ Lemma~\ref{k_local_geodesics}).

This localness translates into an important combinatorial property, known technically as
{\em finiteness of cone types}. This is the basis of Cannon's theory of hyperbolic groups, and
for the more general theory of automatic groups and structures (see \cite{Epstein_et_al} for more details).

\subsection{Regular languages}
Let $S$ be a finite set, and let $S^*$ denote the set of finite words in the alphabet $S$.
An {\em automaton} is a finite directed graph $\Gamma$ with a distinguished initial vertex, and edges
labeled by elements of $S$ in such a way that each vertex has at most one outgoing edge with
a given label. Some subset of the vertices of $\Gamma$ are called {\em accept states}. A word $w$
is $S^*$ determines a simplicial path in $\Gamma$, by starting at the initial vertex, by
reading the letters of $w$ (from left to right) one by one, and by moving along the corresponding
edge of $\Gamma$ if it exists, or halting if not. Associated to $\Gamma$ there is a subset 
$L \subset S^*$ consisting of precisely those words that can be read in their entirety without
halting, and for which the terminal vertex of the associated path ends at an accept state. One
says that $L$ is {\em parameterized} by (paths in) $\Gamma$.

\begin{definition}
A subset $L \subset S^*$ is a {\em regular language} if there is a finite directed graph
$\Gamma$ as above that parameterizes $L$.
\end{definition}

Note that $\Gamma$ is not part of the data of a regular language, and for any given regular
language there will be many graphs that parameterize it. A language is {\em prefix-closed}
if, whenever $w \in L$, every prefix of $w$ is also in $L$ (the empty word is a prefix of every
word).

\begin{lemma}
If $L$ is prefix-closed and regular, there is a $\Gamma$ parameterizing $L$ for which every vertex
is an accept state.
\end{lemma}
\begin{proof}
If $\Gamma$ is any graph that parameterizes $L$, remove all non-accept vertices and the edges
into and out of them.
\end{proof}

\begin{theorem}[Generating function]\label{regular_language_generating_function}
Let $L$ be a regular language, and for each $n$,
let $L_n$ denote the set of elements of length $n$, and $L_{\le n}$ the set of
elements of length $\le n$. Let $s(t):=\sum |L_n|t^n$ and $b(t):=\sum |L_{\le n}|t^n$ be
(formal) generating functions for $|L_n|$ and $|L_{\le n}|$ respectively. Then $s(t)$ and
$b(t)$ are {\em rational}; i.e.\/ they agree as power series expansions with some ratio of
integral polynomials in $t$.
\end{theorem}
\begin{proof}
Note that $b(t)=s(t)/(1-t)$ so it suffices to prove the theorem for $s(t)$. 
Let $\Gamma$ parameterize $L$, and let $M$ be the adjacency matrix of $\Gamma$; 
i.e.\/ $M_{ij}$ is equal to the
number of edges from vertex $i$ to vertex $j$. Let $v_0$ be the vector with a $1$ in the initial
state, and $0$ elsewhere, and let $v_a$ be the vector with a $1$ in every accept state, and $0$
elsewhere. Then $|L_n| = v_0^TM^nv_a$.

A formal power series $A(t):=\sum a_nt^n$
is rational if and only if its coefficients satisfy a linear recurrence; i.e.\/ if there
are constants $c_0,\cdots,c_d$ (not all zero) so that $c_0a_n + c_1a_{n-1} + \cdots + c_da_{n-d}=0$
for all $n\ge d$. For, $A(t)(c_0+c_1t+\cdots + c_dt^d)$ vanishes in degree $\ge d$,
and is therefore a polynomial (reversing this argument proves the converse).

If $p(t)=\sum p_it^{d-i}$ is the characteristic polynomial of $M$, 
then $p(M)=0$, and
$$0 = v_0^TM^{n-d}p(M)v_a= p_0|L_n| + p_1|L_{n-1}| + \cdots + p_d|L_{n-d}|$$
proving the theorem.
\end{proof}

Another way of expressing $s(t)$, more useful in some ways, is as follows.

\begin{proposition}
Let $L$ be a regular language. Then there is an integer $D$ so that for each
value of $n$ mod $D$ either $|L_n|$ is eventually zero, or
there are finitely many constants $\lambda_i$ 
and polynomials $p_i$ so that
$|L_n|=p_1(n)\lambda_1^n + \cdots + p_k(n)\lambda_k^n$
for all sufficiently large $n$.
\end{proposition}

For a proof see e.g.\/ \cite{Flajolet_Sedgewick} Thm.~V.3. In particular,
either $|L_{\le n}|$ has polynomial growth, or $C^{-1}(n^k\lambda^n)\le |L_{\le n}|\le C(n^k\lambda^n)$
for some real $\lambda$ and integer $k$, and constant $C$.

\subsection{Cannon's theorem}

Let $S$ be a set. A total order $\prec$ on $S$ extends to a unique
lexicographic (or dictionary) order on $S^*$ as follows:
\begin{enumerate}
\item{the empty word precedes everything;}
\item{if $u$ and $v$ are both nonempty and start with different letters $s,t\in S$ then
$u \prec v$ if and only if $s \prec t$; and}
\item{if $u \prec v$ and $w$ is arbitrary, then $wu \prec wv$.}
\end{enumerate}
If $G$ is a group and $S$ is a generating set for $G$, there are finitely many geodesic words representing
any given element; the lexicographically first geodesic is therefore a canonical representative
for each element of $g$, and determines a language $L \subset S^*$ that bijects with $G$
under evaluation. We denote evaluation by overline, so if $u\in S^*$, we denote the
corresponding element of $G$ by $\overline{u}$. We similarly denote length of an element of
$S^*$ by $|\cdot|$. So we always have $|\overline{u}|\le |u|$ with
equality if and only if $u$ is geodesic.

Given $g\in G$ the {\em cone type} of $g$, denoted $\cone(g)$, is the set of $h\in G$
for which some geodesic from $\id$ to $gh$ passes through $g$. 
For any $n$, the {\em $n$-level} of $g$ is the set of $h$ in the ball $B_n(\id)$ 
such that $|gh|<|g|$. Cannon showed that the $n$ level (for $n$ sufficiently large)
determines the cone type, and therefore that there are only {\em finitely many
cone types}.

\begin{lemma}[Cannon \cite{Cannon} Lem.~7.1 p.~139]\label{bounded_level_determines_cone_type}
The $2\delta+1$ level of an element determines its cone type.
\end{lemma}
\begin{proof}
Let $g$ and $h$ have the same $2\delta+1$ level, and let $u,v$ be geodesics with
$\overline{u}=g$ and $\overline{v}=h$. Only $\id$ has an empty $2\delta+1$ level, so
we may assume $u,v$ both have length $\ge 1$. We prove the lemma by induction.
Suppose $uw$, $vw$ and $uws$ are geodesics, where $s\in S$. We must show that $vws$ is
a geodesic. Suppose to the contrary that there is some $\overline{w_1w_2} = \overline{vws}$
where $|w_1|=|v|-1$ and $|w_2|\le |w|+1$. Then $h^{-1}\overline{w}_1$ is in the
$2\delta+1$ level of $h$, which agrees with the $2\delta+1$ level of $g$, and therefore
$|gh^{-1}\overline{w}_1|<|g|$. But then concatenating a geodesic representative of
$gh^{-1}\overline{w}_1$ with $w_2$ gives a shorter path to $\overline{uws}$, certifying
that $uws$ is not geodesic, contrary to assumption.
\end{proof}

\begin{figure}[ht]
\centering
\labellist
\pinlabel $g$ at 70 107
\pinlabel $h$ at 303 82
\pinlabel $w_2$ at 250 160
\pinlabel $w_2$ at 14 185
\pinlabel $\overline{vws}$ at 343 245
\pinlabel $\overline{uws}$ at 107 270
\small\hair2pt
\endlabellist
\includegraphics[scale=0.5]{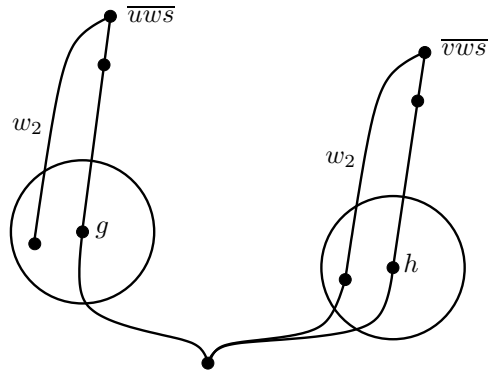}
\caption{A shortcut $w_2$ from the $2\delta+1$ level of $h$ to $\overline{vws}$ gives
a shortcut from the $2\delta+1$ level of $g$ to $\overline{uws}$. This
figure is adapted from \cite{Epstein_et_al}.}
\label{fig:shortcut}
\end{figure}
\medskip

The following theorem is implicit in \cite{Cannon}, though expressed there in somewhat different
language. 

\begin{theorem}[Cannon \cite{Cannon}]\label{Cannon_theorem}
Let $G$ be a hyperbolic group, and $S$ a symmetric generating set. 
Fix a total order $\prec$ on $S$.
Then the language of lexicographically first geodesics is prefix-closed and regular.
\end{theorem}
\begin{proof}
That this language is prefix-closed is obvious. We show it is regular by describing
an explicit parameterizing graph. 

As a warm up, we show first that the language of {\em all} geodesics is regular.
A parameterizing graph can be taken as follows. The vertices (all accept states)
are precisely the set of
cone types, and there is an edge labeled $s$ 
from a cone type of the form $\cone(g)$ to one of the form $\cone(gs)$ 
whenever $|gs| = |g| + 1$. By the definition of cone types, this is well-defined.
By Lemma~\ref{bounded_level_determines_cone_type}, the number of cone types is
finite, so this is a finite graph. By construction, this graph exactly
parameterizes the language of all geodesics.

Now fix a total order $\prec$ on $S$. For each $g\in G$, 
let $u_g$ be the lexicographically
first geodesic from $\id$ to $g$.
For each $g\in G$ a {\em competitor} of $g$
is some $h$ with $|h|=|g|$, with $u_h \prec u_g$, and for which 
$d(\overline{u_h|_{\le i}},\overline{u_g|_{\le i}})\le 2\delta$ for all $i$, where
$u_g|_{\le i}$ denotes the prefix of $u_g$ of length $i$, and similarly for
$u_h|_{\le i}$ (this is described by saying that {\em $u_h$ synchronously fellow-travels
$u_g$}). 

If there is some $g'$ with $|g'| = |g| + d(g,g')$ and
$|g'| = |h| + d(h,g')$ then by $\delta$-thinness and the definition of
geodesics, $u_h$ synchronously fellow-travels $u_g$.
It follows that for all $g\in G$ and $s\in S$ we have $u_{gs}=u_gs$ if and only if
$u_gs$ is a geodesic, and there is no competitor $h$ of $g$ and $s'\in S$
so that $hs'=gs$.

\medskip

Given $g\in G$ define $C(g) \subset B_{2\delta}(\id)$ to be the set of
$h$ for which $gh$ is a competitor of $g$. Associated to $g$ is the list $L(g)$
of pairs $(h\in C(g),\cone(gh))$ together with the cone type of $g$ itself.
Note that the set of possible lists $L(g)$ is {\em finite}. We can now define
a parameterizing graph by taking the vertices (all accept states) to be the
possible lists $L(g)$, and there is an edge labeled $s$ from a list of the
form $L(g)$ to a list of the form $L(gs)$ if and only if
$|gs| = |g| + 1$, and there is no $h\in C(g)$ and $s'\in S$ with $ghs'=gs$.
This is evidently a finite directed graph, which parameterizes the $u_g$; we must
show it is well-defined.

First of all, $h \in C(gs)$ if and only if one of the two following possibilities occurs:
\begin{enumerate}
\item{there is some $h'\in C(g)$ and $s'\in S\cap \cone(gh')$ with $gh's'=gsh$; or}
\item{there is some $s'\prec s$ in $S\cap \cone(g)$ with $gs'=gsh$.}
\end{enumerate}
Both of these possibilities depend only on $C(g)$, $\cone(g)$ or $\cone(gh')$ for some
$h' \in C(g)$, and not on $g$ itself.
Second of all, if $h\in C(gs)$, then $\cone(gsh)$ depends only on $\cone(gh')$ and $s'$
in the first case, and on $\cone(g)$ and $s'$ in the second case. This shows the
graph is well-defined, and completes the proof of the theorem.
\end{proof}

\begin{remark}
This completes the proof of Lemma~\ref{zeta_diverges}, and the subsequent results in
\S~\ref{Patterson_Sullivan_subsection}. The reader will note that the results in this
section do not depend on Lemma~\ref{zeta_diverges}, so our reasoning has not been circular.
\end{remark}

\subsection{Combings and combable functions}

\begin{definition}
Let $G$ be a group, and $S$ a generating set. A {\em combing} for $G$ (with
respect to $S$) is
a prefix-closed regular language $L\subset S^*$ which bijects with $G$ under
evaluation, and satisfies $|\overline{u}|=|u|$ for all $u\in L$ (i.e.\/ $L$ is
a language of geodesics).
\end{definition}

Theorem~\ref{Cannon_theorem} says that every hyperbolic group admits a combing.
If $L$ is a combing with respect to $S$, the {\em $L$-cone type} of $g$,
denoted $\cone_L(g)$, is the set of $h\in G$ for which the $L$-geodesic
evaluating to $gh$ contains a prefix (which is also an $L$-geodesic) evaluating
to $g$. There is a graph $\Gamma$ parameterizing $L$ with one vertex
for each $L$-cone type, and an edge from $\cone_L(g)$ to $\cone_L(gs)$ 
labeled $s$ whenever $s\in\cone_L(g)$.

\begin{remark}
The reader should be warned that many competing definitions of combing 
exist in the literature. 
\end{remark}

Suppose $L$ is a combing of $G$, and $\Gamma$ is a graph parameterizing
$L$, so that there is a (length-preserving) bijection between directed
paths in $\Gamma$ starting at the initial vertex, and words of $L$,
by reading the edge labels of the path. If $u\in L$, we let $\gamma(u)$ denote
the corresponding path in $\Gamma$, and $\gamma(u)_i$ the successive vertices
in $\Gamma$ visited by $\gamma(u)$.

\begin{definition}
A function $\phi:G \to \Z$ is {\em weakly combable} with respect
to a combing $L$ if there is a graph $\Gamma$ parameterizing $L$ and
a function $d\phi$ from the vertices of $\Gamma$ to $\Z$ so that
$\phi(\overline{u}) = \sum_i d\phi(\gamma(u)_i)$ for all $u\in L$.

A function $\phi$ is {\em combable} if it is weakly combable with respect
to some combing $L$, and if there is a constant $C$ so that 
$|\phi(gs)-\phi(g)|\le C$ for all $g\in G$ and $s\in S$; and it is
{\em bicombable} if it is combable, and further satisfies
$|\phi(sg) - \phi(g)|\le C$.
\end{definition}

\begin{remark}
It might be more natural to define a function $d\phi$ on the
{\em edges} of $\Gamma$ instead of its vertices; however, associated to
any directed graph $\Gamma$ there is another graph --- the {\em line graph}
of $\Gamma$ --- whose vertices are the edges of $\Gamma$, and whose edges
are the composable pairs of edges of $\Gamma$, and the line graph of
$\Gamma$ parameterizes $L$ if $\Gamma$ does.
\end{remark}

\begin{lemma}[Calegari--Fujiwara \cite{Calegari_Fujiwara} Lem.~3.8]
The property of being combable or bicombable does not depend on the choice
of a generating set or a combing.
\end{lemma}
The proof proceeds along the same lines as Theorem~\ref{Cannon_theorem}.
The key point is that words in $L$ are (uniformly) quasigeodesic with
respect to $S'$, and therefore stay within a bounded distance of words in $L'$
with the same evaluation. Therefore an automaton reading the letters
of an $L'$-word can keep track of the states of
a collection of automata simultaneously reading nearby $L$-words, and
keeping track of how $\phi$ changes as one goes along.
See \cite{Calegari_Fujiwara} for details.

\begin{example}\label{word_length_example}
Word length in any generating set is bicombable. In fact, if $S$ is a
(possibly unsymmetric) set which generates $G$ as a semigroup, word length
in $S$ is bicombable. One can generalize word length by giving different generators
(and corresponding edges in the Cayley graph) different lengths; 
providing the lengths are all integral and positive, the resulting (geodesic)
word length is bicombable.
\end{example}

\begin{example}\label{sum_example}
The sum or difference of two (bi)combable functions is (bi)combable.
\end{example}

\begin{example}
The following definition is due to Epstein--Fujiwara \cite{Epstein_Fujiwara} (also
see \cite{Brooks}).
Let $\sigma$ be a path in $C_S(G)$. A {\em copy} of $\sigma$ is a translate
$g\sigma$ for some $g\in G$. Given a path $\gamma$ in $c_S(G)$, define
$c_\sigma(\gamma)$ to be the maximal number of disjoint
copies of $\sigma$ in $\gamma$, and
for $g\in G$ define the {\em small counting function} $c_\sigma:G \to \Z$ by
the formula
$$c_\sigma(g)=|g| - \inf_\gamma(|\gamma|-c_\sigma(\gamma))$$

Counting functions are bicombable. In fact, we can add $\sigma$ to $S$
as a (semigroup) generator, but insist that the (directed) edges labeled
$\sigma$ have length $|\sigma|-1$ instead of $1$; this defines a new distance
function $|\cdot|_\sigma$ which is bicombable (by Example~\ref{word_length_example}), 
and therefore so is the difference $|\cdot| - |\cdot|_\sigma = c_\sigma$
(by Example~\ref{sum_example}).
\end{example}

Many variations on this idea are possible; for instance, the ``big'' counting
functions $C_\sigma$ which count {\em all} copies of $\sigma$ in $\gamma$,
not just the maximal number of disjoint copies.

\subsection{Markov chains}\label{Markov_chain_subsection}

A directed graph $\Gamma$ is sometimes called a {\em topological Markov chain}. A topological
Markov chain can be promoted to a genuine (stationary) Markov chain by assigning
probabilities to each edge in such a way that the probabilities on the
edges leaving each vertex sum to $1$. Recall that we write the adjacency matrix
as $M$; we think of this as an endomorphism of the vector space $V$ spanned by the
states of $\Gamma$. Let $\1$ denote the vector with all components equal to $1$, and
let $\iota$ denote the vector corresponding to the initial state.

Two states in a topological Markov chain are 
said to be {\em communicating} if there is a directed
path from each to the other. The property of being communicating is an equivalence relation.
We write $C_1 \to C_2$ for equivalence classes $C_1$ if there is a directed path from some (any)
vertex of $C_1$ to some (any) vertex of $C_2$; observe that $\to$ is a partial order.
We call each equivalence class a {\em component}. 

The induced (directed) subgraph associated to a component $C$ is itself a topological
Markov chain. Its adjacency matrix $M_C$ has the property that for any $i$ and $j$ there is an
$n$ (in fact, infinitely many $n$) so that $(M_C^n)_{ij}$ is positive; one says such a
Markov chain is {\em irreducible}. If there is a fixed $n$ so that $(M_C^n)_{ij}$ is 
positive for all $i$, $j$ we say the Markov chain is {\em aperiodic}; this holds exactly
when the gcd of the lengths of all loops in $C$ is $1$. A Markov chain (on a finite
state space) which is both irreducible and aperiodic is {\em ergodic}.

\begin{lemma}[Perron--Frobenius]\label{Perron_Frobenius}
Let $M$ be a real matrix with positive entries. 
Then there is a unique eigenvalue $\lambda$ of biggest
absolute value, and this eigenvalue is real and positive. Moreover, $\lambda$ is
a simple root of the characteristic polynomial, and it has a right (left) eigenvector
with all components positive, unique up to scale. 
Finally, any other non-negative right (left) eigenvector is a multiple of the $\lambda$
eigenvector.
\end{lemma}
\begin{proof}
Since the entries of $M$ are positive, $M$ takes the positive orthant
strictly inside itself. The projectivization of the positive orthant is a simplex,
and therefore $M$ takes this simplex strictly into its interior. It follows that
$M$ has a unique attracting fixed point in the interior this simplex; 
this fixed point corresponds to the unique eigenvector $v$ (up to scale) with non-negative entries, 
and its entries are evidently all positive, and its associated eigenvalue $\lambda$
is real and positive.

If $\pi$ is any plane containing this unique positive eigenvector, the projectivization
of $\pi$ is an $\RP^1$; since the eigenvector becomes an attracting fixed point in this
$\RP^1$, it is not the only fixed point. This shows that $\lambda$ is a simple eigenvalue;
a similar argument shows that $-\lambda$ is not an eigenvalue.

Let $\mu$ be any other eigenvalue. If $\mu$ is real, then $|\mu|<\lambda$. Suppose
$\mu$ is complex, acting as composition of a dilation with a rotation on some plane
$\pi$. If $|\mu|=\lambda$ then the restriction of $M$ to $\pi \oplus \langle v \rangle$
acts projectively like a rotation; but this contradicts the fact that $v$ is 
a projective attracting fixed point. This proves the theorem.
\end{proof}

If $M$ is non-negative, there is still a non-negative real eigenvector $v$
with a real positive eigenvalue $\lambda$, and every other eigenvalue $\mu$ satisfies
$|\mu|\le \lambda$. In this generality, $\lambda$ might have multiplicity $>1$, and the Jordan
block associated to $\lambda$ might not be diagonal. However if $M$ is {\em irreducible},
then $\lambda$ has multiplicity $1$, the eigenvector $v$ is strictly positive, 
and every other eigenvalue with absolute value $\lambda$ is simple and
of the form $e^{2\pi i/k}\lambda$. These facts can be proved similarly to the 
proof of Lemma~\ref{Perron_Frobenius}

\medskip 

Now let $G$ be a hyperbolic group, $L$ a combing with respect to some generating set,
and $\Gamma$ a graph parameterizing $L$. Let $\Gamma_C$
be the quotient directed graph whose vertices are the components of $\Gamma$.
Note that $\Gamma_C$ contains no directed loops.
Associated to each vertex of $\Gamma_C$ is an adjacency matrix $M_C$ which
has a unique maximal real eigenvalue $\lambda(C)$ of multiplicity $1$. We let
$\lambda = \max_C \lambda(C)$, and we call a component {\em maximal} if $\lambda(C)=\lambda$.

The next lemma is crucial to what follows, and depends on Coornaert's estimate
of the growth function (i.e.\/ Corollary~\ref{Coornaert_growth_rate}).

\begin{lemma}\label{maximal_components_parallel}
The maximal components do not occur in parallel; that is, there is no directed
path from any maximal component to a distinct maximal component.
\end{lemma}
\begin{proof}
Since there is a directed path from the initial vertex to every other vertex,
the number of paths of length $n$ is of the form $p(n)\lambda^n+O(q(n)\xi^n)$
for polynomials $p$, $q$ and $\xi<\lambda$, where $\lambda$ is as
above. Moreover, the degree of $p$ is one less
than the length of the biggest sequence of maximal components 
$C_0 \to C_1 \to \cdots \to C_{\text{deg}(p)}$.
The number of paths of length $n$ is equal to the number of elements of $G$ of
length $n$, so Corollary~\ref{Coornaert_growth_rate} implies that
the degree of $p$ is zero.
\end{proof}

It follows that all but exponentially few paths $\gamma$ of length $n$ in $\gamma$
are entirely contained in one of the maximal components of $\Gamma$, except for a
prefix and a suffix of length $O(\log(n))$. Consequently, the properties of
a ``typical'' path in $\Gamma$ can be inferred from the properties of a ``typical''
path conditioned to lie in a single component.

For any vector $v$, the limits 
$$\rho(v):=\lim_{n\to\infty} n^{-1} \sum_{i=0}^{n-1}
\lambda^{-i}M^iv, \quad \ell(v):=\lim_{n\to\infty} n^{-1} \sum_{i=0}^{n-1} \lambda^{-i}(M^T)^iv$$
exist, and are the projections onto the left and right $\lambda$-eigenspaces respectively. 
Heuristically, $\ell(v)$ is the distribution of endpoints of long
paths that start with distribution $v$, and $\rho(v)$ is the distribution of
starting points of long paths that end with distribution $v$.

Recall that $\iota$ denotes the vector with a $1$ in the coordinate corresponding
to the initial vertex and $0$s elsewhere, 
and $\1$ denotes the vector with all coordinates equal to $1$.
Define a measure $\mu'$ on the vertices of $\Gamma$ by $\mu'_i = \ell(\iota)_i\rho(\1)_i$, 
and scale $\mu'$ to a probability measure $\mu$.
Define a matrix $N$ by $N_{ij}=M_{ij}\rho(\1)_j/\lambda\rho(\1)_i$ if
$\rho(\1)_i\ne 0$, and define $N_{ij}=\delta_{ij}$ otherwise.

\begin{lemma}\label{N_preserves_mu}
The matrix $N$ is a stochastic matrix (i.e.\/ it is non-negative, and the rows
sum to $1$) and preserves the measure $\mu$.
\end{lemma}
\begin{proof}
If $\rho(\1)_i=0$ then $\sum_j N_{ij}=1$ by fiat. Otherwise
$$\sum_j N_{ij} = \sum_j \frac {M_{ij}\rho(\1)_j} {\lambda\rho(\1)_i} = \frac{(M\rho(\1))_i}
{\lambda\rho(\1)_i}=1$$.
To see that $N$ preserves $\mu'$ (and therefore $\mu$), we calculate
$$\sum_i \mu'_i N_{ij} =\sum_i \ell(\iota)_i\rho(\1)_i \frac{M_{ij}\rho(\1)_j}{\lambda\rho(\1)_i}
= \sum_i \frac{\ell(\iota)_iM_{ij}}{\lambda} \rho(\1)_j = \ell(\iota)_j\rho(\1)_j=\mu'_j$$
\end{proof}

In words, $\mu_i$ is the probability that a point on a path will be in state $i$, conditioned
on having originated at the initial vertex in the distant past, and conditioned on having
a distant future.

\subsection{Shift space}

For each $n$ let $Y_n$ denote the set of paths in $\Gamma$ of length $n$
starting at the initial vertex, and let $X_n$ denote the set of {\em all} paths
in $\Gamma$ of length $n$. We can naturally identify $X_0$ with the vertices of
$\Gamma$.

Restricting to an initial subpath defines an inverse system 
$\cdots \to X_n \to \cdots \to X_1 \to X_0$, and the inverse limit $X_\infty$ is
the space of (right) infinite paths. Similarly define $Y_\infty \subset X_\infty$.
If we give each $X_n$ and $Y_n$ the discrete topology, then $X_\infty$ and $Y_\infty$
are Cantor sets.

If $x:=x_0,x_1,\cdots$ and $x':=x_0',x_1'\cdots$ are two elements of $X_\infty$, 
we define $(x|x')$ to be the first index at which $x$ and $x'$ differ, 
and define a metric on $X_\infty$ by setting $d(x,x')=a^{-(x|x')}$ for some $a>1$
(the notation $(\cdot|\cdot)$ is deliberately intended to suggest a resemblance
to the Gromov product).
If we like, we can define $\overline{X}=\cup_i X_i \cup X_\infty$
and metrize it (as a compact space, in which each $X_n$ sits as a discrete subset)
in the same way. Similarly, give $Y_\infty$ the induced metric, and define
$\overline{Y}=\cup_i Y_i \cup Y_\infty$ likewise.

The shift operator $T:X_\infty \to X_\infty$ is defined by $(Tx)_i=x_{i+1}$.
We define a probability measure $\mu$ on each $X_n$ by 
$\mu(x_0\cdots x_n)=\mu_{x_0}N_{x_0x_1}N_{x_1x_2}\cdots N_{x_{n-1}x_n}$ where
$\mu$ and $N$ are the measure and stochastic matrix whose properties are
given in Lemma~\ref{N_preserves_mu}. 
By the definition of an inverse limit, there is a map $X_\infty \to X_n$ for each $n$
which takes an infinite path to its initial subpath of length $n$; the preimages of
subsets of the $X_n$ under such maps are a basis for the topology on $X_\infty$, called
{\em cylinder sets}. The measures $\mu$ as above let us define a Borel probability measure
$\mu$ on $X_\infty$ by first defining it on cylinder sets (note that the definitions of $\mu$
on different $X_n$ are compatible) and extending it to all Borel sets in the standard way;
Lemma~\ref{N_preserves_mu} implies that $\mu$ is $T$-invariant (i.e.\/ $\mu(A)=\mu(T^{-1}(A))$
for all measurable $A\subset X_\infty$).

There is a bijection between $Y_n$ and $L_n$, and by evaluation with $G_n$.
This map extends continuously to a map $E:\overline{Y} \to \overline{G}$, by sending
$Y_\infty \to \partial_\infty G$. 

\begin{lemma}
The map $E:\overline{Y} \to \overline{G}$ is surjective, Lipschitz in the $a$-metric,
and bounded-to-one.
\end{lemma}
\begin{proof}
That the map is Lipschitz follows immediately from the definition, and the
observation that $(E(y)|E(y'))\le (y|y')-\delta$ for $y,y'\in \overline{Y}$.
The restrictions $E:Y_n \to G_n$ are all bijections, so we just need to check that
$Y_\infty \to \partial_\infty G$ is surjective and bounded-to-one.

Since $E$ is continuous, $\overline{Y}$ is compact and $\overline{G}$ is Hausdorff, 
the image is compact. Since the image is dense (because it contains $G_n$ for all $n$), it is
surjective. 

Finally, observe that if $y$ and $y'$ are any two points in $Y_\infty$, and
$\gamma,\gamma'$ are the associated infinite geodesics in $G$, then $\gamma \cap \gamma'$
is a compact initial segment, since after they diverge they never meet again (by the
definition of a combing). Fix $x\in \partial_\infty G$, let $y_i$ be a finite
subset of $E^{-1}(x)$, and let $\gamma_i$ be the
geodesic rays in $G$ corresponding to the $y_i$. For all but finitely many
points $p$ on any $\gamma_i$, each $\gamma_j$ 
intersects the ball $B_\delta(p)$ {\em disjointly} from the others. In particular,
the number of points in the preimage of any point in $\partial_\infty G$ is bounded
by the cardinality of a ball (in $G$) of radius $\delta$.
\end{proof}

Recall that in \S~\ref{Patterson_Sullivan_subsection} we defined probability
measures $\nu_s$ on $G$ for each $s>h(G)$. Note that $e^{h(G)}=\lambda$ where
$\lambda$ is as above. For each $n$ we define a probability
measure on $Y_n$ (which, by abuse of notation, we call $\nu_s$) by
$\nu_s(y) = \nu_s(E(y)\cone_L(E(y)))$ for $y\in Y_n$, and observe that the limit
as $s \to h(G)$ from above (which we denote $\nu(y)$) exists and depends only on the cone
type $\cone_L(E(y))$. Since the Patterson--Sullivan measure
$\nu$ is supported in $\partial_\infty G$, the measures $\nu$ on each $Y_n$ are compatible,
thinking of each $Y_n$ as a collection of cylinder sets in $Y_\infty$, and define a unique probability
measure $\nu$ on $Y_\infty$ which pushes forward under $E$ to $\nu$ on $\partial_\infty G$.

\begin{lemma}\label{nu_to_mu}
The measure $\mu$ on $X_\infty$ is the limit $\mu = \lim_{n\to\infty} \frac 1 n \sum_{i=0}^{n-1} T^i_*\nu$.
\end{lemma}
\begin{proof}
We give the sketch of a proof. For any $y$, let $L^y$ be the (regular)
language of suffixes of words in $L$ with $y$ as a prefix, and let $L^y_n$
be the subset of $L^y$ of length $n$. Then
$\nu_s(y) = \zeta_G^{-1}(s)\sum_n e^{-s(|y|+n)}|L^y_n|$.

If there is no path from
the final state $y_n$ to a maximal component, the growth rate of $L^y$
is strictly less than that of $L$, and $\nu_s(y) \to 0$. Otherwise both
growth functions are eventually of the form $C\lambda^n$ plus something
exponentially small compared to $\lambda^n$. Define measures $\nu_m$
on $Y_n$ by $\nu_m(y) = \frac 1 m \sum_{i=1}^m \lambda^{-(|y|+i)}|L^y_i|$.
Then by considering the form of the growth functions of $L$ and $L^y$, we see that
there is a constant $C$ (not depending on $y$ or $n$)
so that $\lim_{m \to \infty} \nu_m(y) = C\nu(y)$. Scaling $\nu_m$ to be
a probability measure, we can set $C=1$.

The proof now follows from the definition of $\mu,N$; 
see \cite{Calegari_Fujiwara} Lem.~4.19 for details. 
\end{proof}

\subsection{Limit theorems}

Let $\xi_1,\xi_2,\cdots$ be a (stationary) irreducible Markov chain on a finite
state space, with stationary measure $\mu$,
and let $f$ be a real-valued function on the state space (since this
space is finite, there are no additional assumptions on $f$; in general we require
$f$ to be integrable, and have finite variance). Define $F_n:=\sum_{i=1}^n f(\xi_i)$,
and $A=\int fd\mu$.

\begin{theorem}[Markov's central limit theorem]\label{Markov_clt}
With notation as above, there is some 
$\sigma\ge 0$ so that for any $r\le s$,
$$\lim_{n\to\infty} \Pr\left(r\le \frac {F_n-nA}{\sigma\sqrt{n}} \le s\right) = \frac 1 {\sqrt{2\pi}}\int_r^s e^{-x^2/2} dx$$
\end{theorem}

Equivalently, there is convergence in probability
$n^{-1/2}(F_n-nA) \to N(0,\sigma)$ where $N(0,\sigma)$ denote the normal distribution
with mean $0$ and standard deviation $\sigma$ (in case $\sigma=0$, we let $N(0,\sigma)$
denote a Dirac mass centered at $0$).

\medskip

Now, each maximal component $C$ as above is a stationary irreducible Markov chain,
with stationary measure the conditional measure $\mu|C$. The measure $\mu$ on
$X_\infty$ decomposes measurably into the union of (shift-invariant) 
subspaces $X_\infty(C)$, the subspace of (right) infinite
sequences contained in the component $C$. Consequently, if $\phi$ is
a combable function on $G$, then for each maximal component $C$, there are
constants $A_C = \int_C d\phi/\mu(C) d\mu$ and $\sigma_C$, so that for $\mu$-a.e.
$x \in X_\infty(C)$, the random variable $n^{-1/2}(\sum_{i=0}^{n-1} d\phi(x_i) - nA_C)$
converges in probability to $N(0,\sigma_C)$.

By Lemma~\ref{nu_to_mu}, for $\nu$-a.e. $y\in Y_\infty$ there
is a unique $C$ so that $T^ny \in X_\infty(C)$ for sufficiently big $n$;
we say that $y$ is {\em associated} to the component $C$. Let $Y_\infty(C)$
be the set of $y$ associated to a fixed $C$. For $\nu$-a.e. $y\in Y_\infty(C)$
we have convergence in probability
$n^{1/2}(\sum_{i=0}^{n-1} d\phi(y_i) - nA_C) \to N(0,\sigma_C)$
(one way to see this is to observe that this is a shift-invariant tail property of $y$,
and use Lemma~\ref{nu_to_mu}).

For combable functions, this is the end of the story. It is certainly possible for the
constants $A_C,\sigma_C$ to vary from component to component. But for
bicombable $\phi$ we have the following key lemma:

\begin{lemma}\label{common_variance}
Let $\phi$ be bicombable. Then there are constants $A,\sigma$ so that $A_C=A$ and
$\sigma_C=\sigma$ for all maximal components $C$.
\end{lemma}
\begin{proof}
Call $y\in Y_\infty$ {\em typical} if there are constants $A_y$ and $\sigma_y$
(necessarily unique) so that 
$n^{1/2}(\sum_{i=0}^{n-1} d\phi(y_i) - nA_y) \to N(0,\sigma_y)$. For each $C$ we
have seen that $\nu$-a.e. $y\in Y_\infty(C)$ is typical with $A_y=A_C$ and $\sigma_y=\sigma_C$.

The map $E:Y_\infty \to \partial_\infty G$ is finite-to-one, and takes the measure
$\nu$ on $Y_\infty$ to the Patterson--Sullivan measure $\nu$ on $\partial_\infty G$.
Hence $E(Y_\infty(C))$ has positive measure for each $C$. Let $y\in \partial_\infty G$ 
be typical, and let $\id,g_1,g_2,\cdots$ be the associated geodesic
sequence of elements in $G$ converging to $E(y)$. Now let $g$ be
arbitrary, let $y'$ be any element of $Y_\infty$ with $E(y')=gE(y)$, and let
$\id,g_1',g_2',\cdots$ be the geodesic sequence of elements in $G$ associated to $y'$.
By $\delta$-thinness, $d(g_i',gg_i)$ is eventually approximately constant, and therefore
bounded. Since $\phi$ is bicombable, $y'$ is typical, with $A_{y'}=A_y$ and
$\sigma_{y'}=\sigma_y$. But the action of $G$ on $\nu$ is ergodic for 
$\nu$, by Corollary~\ref{ergodic_at_infinity}, and therefore for any $C,C'$ there
are typical $y\in Y_\infty(C)$, $y'\in Y_\infty(C')$ with $A_y=A_C,\sigma_y=\sigma_C$
and $A_{y'}=A_{C'},\sigma_{y'}=\sigma_{C'}$, and with $y'=gy$ for some $g$. This
completes the proof.
\end{proof}

\begin{corollary}[Calegari--Fujiwara \cite{Calegari_Fujiwara}]
Let $G$ be hyperbolic, and let $\phi$ be bicombable. Then there are constants $A,\sigma$
so that if $g_n$ denotes a random
element of $G_n$ (in the $\nu$ measure), there is convergence in probability
$n^{-1/2}(\phi(g_n)-nA) \to N(0,\sigma)$.
\end{corollary}

Note that $A$ and $\sigma$ as above are {\em algebraic}, and one can estimate from
above the degree of the field extension in which they lie from the complexity
of $\Gamma$.

The uniform measure and the measure $\nu$ on $G_n$ are uniformly quasi-equivalent
on a large scale, in the sense that there are constants $R$ and $C$ so that
for any $g\in G_n$, there is an inequality
$$C^{-1} |B_R(g)\cap G_n|/|G_n| \le \nu(B_R(g)\cap G_n) \le C|B_R(g)\cap G_n|/|G_n|$$
It follows that if $g_n$ denotes a random element of $G_n$
(in the uniform measure), the distribution $n^{-1/2}(\phi(g_n)-nA)$ has a tail that
decays like $C_1e^{-C_2t^2}$.

Since length with respect to one generating set is bicombable with respect to another,
we obtain the following corollary:
\begin{corollary}
Let $G$ be hyperbolic, and let $S$ and $S'$ be two finite generating sets for $G$.
There is an algebraic number $\lambda_{S,S'}$ so that if $g_n$ is a random element of $S$
of word length $n$, then the distribution $n^{-1/2}(|g_n|_{S'}-n\lambda_{S,S'})$ has a tail
that decays like $C_1e^{-C_2t^2}$ when $n$ is sufficiently large.
\end{corollary}

It is a slightly subtle point that $\lambda_{S',S} \ge \lambda_{S,S'}^{-1}$, and
the inequality is strict except for essentially trivial cases.

\subsection{Thermodynamic formalism}

To push these techniques further, we must study classes of functions more general
than combable functions, and invoke more sophisticated limit theorems.
There is a well-known framework to carry out such analysis, pioneered by Ruelle, Sinai,
Bowen, Ratner, Parry etc.; \cite{Ruelle} is a standard reference.

The setup is as follows. For simplicity, let $M$ be a $k\times k$ matrix
with $0$--$1$ entries for which there is a constant $n$ so that all the entries of
$M^n$ are positive (i.e.\/ $M$ is the adjacency matrix of a topological Markov chain
with $k$ states which is irreducible and aperiodic). Let $X_\infty$ be the
space of (right) infinite sequences $x:=x_0,x_1,\cdots$ satisfying $M(x_n,x_{n+1})=1$
for all $n$, and let $T$ be the shift operator on $X_\infty$. As before, we
can metrize $X_\infty$ by $d(x,x')=a^{-(x|x')}$ for some fixed $a>1$, and observe
that the action of $T$ on $X_\infty$ is {\em mixing}. This means that for all nonempty
open sets $U,V\subset X_\infty$ there is $N$ so that $T^{-n}(U)\cap V$ is nonempty
for all $n\ge N$. Note that if $M$ is irreducible but not aperiodic, there is 
nevertheless a decomposition of $X_\infty$ into $D$ 
disjoint components which are cycled by $T$, and 
such that $T^D$ is mixing on each component, where $D$ is the gcd of the periods 
of $T$-invariant sequences.

Let $\M_T$ be the space of $T$-invariant probability measures on $X_\infty$. This
is a convex, compact subset of the space of all measures in the weak-$^*$ topology.
It is not hard to show that the topological entropy $h$ of $T$ is equal to 
the supremum of the measure theoretic entropies $\sup_{\mu\in \M_T} h(\mu)$, and that
$h=\log{\lambda}$ where $\lambda$ is the Perron--Frobenius eigenvalue of $M$;
see e.g.\/ \cite{Ruelle}. 

The shift $T$ uniformly expands $X_\infty$ by a factor of $a$, 
and therefore if a function on
$X_\infty$ is sufficiently regular, it tends to be smoothed out by $T$. Define
$T^*f$ by $T^*f(x) = f\circ Tx$. We would like the iterates $(T^n)^*f$ to have a
uniform modulus of continuity; this is achieved precisely by insisting that $f$
be H\"older continuous, that is, that there is some $\alpha$ so that 
$|f(x)-f(x')|\le C d(x,x')^\alpha = C a^{-\alpha(x|x')}$. The set of 
functions $f$ on $X_\infty$, H\"older continuous of exponent $\alpha$, is a Banach
space with respect to the norm $\|f\|_\infty + \|f\|_\alpha$ where $\|f\|_\alpha$
is the least such $C$ so that $|f(x)-f(x')|\le C d(x,x')^\alpha$. We denote this
Banach space $C^\alpha(X_\infty)$.

\begin{definition}
Let $f$ be H\"older continuous on $X_\infty$. The {\em pressure} of $f$, denoted
$P(f)$, is $P(f) = \sup_{\mu \in \M_T} (h(\mu) + \int f d\mu)$.
\end{definition}

It turns out that the supremum is realized on some invariant measure $\mu_f$ 
of full support, known as the {\em equilibrium state} (or {\em Gibbs state}) 
of $f$. That is, $P(f) = h(\mu_f) + \int f d\mu_f$. See e.g.\/ \cite{Bowen} 
Ch.~1 for a proof of this theorem, and of Theorem~\ref{RPF_theorem} below.

\begin{definition}
The {\em Ruelle transfer operator} $L_f$ associated to $f$ is defined by the
formula $L_f g(x) = \sum_{Tx'=x} e^{f(x')}g(x')$. Note that $L_f$ acts as a bounded
linear operator on $C^\alpha(X_\infty)$.
\end{definition}

\begin{theorem}[Ruelle--Perron--Frobenius \cite{Ruelle}]\label{RPF_theorem}
The operator $L_f$ has a simple positive eigenvalue $e^{P(f)}$ which is strictly
maximal in modulus. The essential spectrum is contained in a ball whose radius
is strictly less than $e^{P(f)}$, and the rest of the 
spectrum outside this ball is discrete and consists of genuine eigenvalues.

There is a strictly positive eigenfunction $\psi_f$ satisfying $L_f\psi_f = e^{P(f)}\psi_f$,
and an ``eigen probability measure'' $\nu_f$ satisfying $L_f^*\nu_f = e^{P(f)}\nu_f$,
and if we scale $\psi_f$ so that $\int \psi_f d\nu_f=1$, then
the equilibrium state $\mu_f$ is equal to $\nu_f \psi_f$.
\end{theorem}

\begin{remark}
$\nu_f$ can be thought of as a left eigenvector for $L_f$, and $\psi_f$ as a right
eigenvector. When $f$ is identically zero, $L_f$ is basically just the matrix
$M$, and $\mu_f$ is basically just $\mu$ as constructed in \S~\ref{Markov_chain_subsection}.
\end{remark}

Pollicott \cite{Pollicott_CRPF} proved a complexified version of 
the RPF theorem, and showed that
$P(f)$ and $\psi_f$ are {\em analytic} on an open subset of
the {\em complex} Banach space $C^\alpha(X_\infty,\C)$ which contains a neighborhood
of $C^\alpha(X_\infty,\R)$ (i.e.\/ of $C^\alpha(X_\infty)$).

Because of the simplicity and analyticity of the maximal eigenvector/value, one can
study the {\em derivatives} of pressure. For simplicity, let $P(t):=P(tf+g)$. Then
we can compute
$$P'(0) = \int f d\mu_g$$
and a further differentiation gives
$$P''(0) = \int f^2 + 2fw'(0)d\mu_g$$
where $w(t) = \psi_{tf+g}$ (suitably normalized).

Now, let $F_n(x) = \sum_{i=0}^{n-1} f(T^ix)$. Then from the definition of the
transfer operator, $L_{tf+g}^n(\cdot) = L_g^n(e^{tF_n} \cdot)$, and therefore one
obtains
$$nP''(0) = \int F_n^2 + 2F_nw'(0)d\mu_g$$
If we set $g$ to be identically zero, then $\mu_g$ is just the equilibrium measure
$\mu$ from before. If we change $f$ by a constant $f - \int fd\mu$ to have mean $0$,
then the ergodic theorem shows $(1/n)F_n \to 0$ $\mu$-a.e. and therefore
$$P''(0) = \lim_{n \to \infty} \int \frac 1 n F_n^2 d\mu$$
It is usual to denote this limiting quantity by $\sigma^2$.

The analyticity of $P$ lets us control the higher moments of $F_n$ in a uniform
manner, and therefore by applying Fourier transform, one obtains a central limit theorem
$n^{-1/2}F_n \to N(0,\sigma)$. Better estimates of the rate of convergence can be
obtained by studying $P'''(0)$; see \cite{Coelho_Parry}.

\medskip

This theorem can be combined with Lemma~\ref{common_variance} to obtain
a central limit theorem for certain functions on hyperbolic groups whose (discrete)
derivatives along a combing satisfy a suitable H\"older continuity property. Such
functions arise naturally for groups acting cocompactly on $\CAT(K)$ spaces
with $K<0$, where one wants to compare the intrinsic geometry of the space with the
``coarse'' geometry of the group.

Let $Z$ be a complete $\CAT(K)$ geodesic metric space with $K<0$, and let $G$ act
cocompactly on $Z$ by isometries. Pick a basepoint $z\in Z$, and define 
a function $F$ on $G$ by $F(g)=d(z,gz)$. Since $G$ is hyperbolic, if we fix a finite
generating set $S$ we can choose a geodesic combing $L$ with respect to $S$ as
above. Now, for any $s\in S$ define $D_sF(g) = F(g) - F(sg)$. 
It is straightforward to see from the $\CAT(K)$ property
that there are constants $C$ and $\alpha$ (depending on $K$ and $G$) so that 
$|D_sF(g) - D_sF(h)|\le C a^{-\alpha(g|h)}$ for all $s$ and all $g,h\in G$. 

An element of $\cup X_n$ corresponds to a path in $\Gamma$. Reading the edge labels determines
a word in the generators (a suffix of some word in $L$), and by evaluation, an element of $G$.
Let $E:\cup X_n \to G$ denote this evaluation map (note that this is not injective).
We can define a function $DF$ on $\cup X_n$ by $DF(x) = D_s F(E(x))$ where $s^{-1}$ is the
label associated to the transition from $x_0$ to $x_1$ 
(we could suggestively write $s=x_1^{-1}x_0$).
Evidently, $DF$ extends to a H\"older continuous function on $\overline{X}$.
Furthermore, for each $y\in Y_n$, we have $\sum_{i=0}^{n-1} DF(T^iy) = F(E(y))$.

For each big component $C$, it follows that $\nu$-a.e. $y\in Y_\infty(C)$ are
$A_C,\sigma_C$ typical (for the function $DF$) for some $A_C,\sigma_C$ depending only on $C$.
Since $F$ is Lipschitz on $G$ in the left and right invariant metrics, the argument of
Lemma~\ref{common_variance} implies that $A_C,\sigma_C$ are equal to some common 
values $A,\sigma$, and therefore we obtain the following corollary:

\begin{corollary}\label{distance_CLT}
Let $Z$ be a complete $\CAT(K)$ geodesic metric space with $K<0$, and let $G$ act
cocompactly on $Z$ by isometries. Pick a basepoint $z\in Z$, and a finite generating set
$S$ for $G$. Then there are constants $A$ and $\sigma$ so that if $g_n$ is a random
element of $G_n$ (in the $\nu$ measure), there is convergence in probability
$n^{-1/2}(d(z,g_nz) - An) \to N(0,\sigma)$.
\end{corollary}

Evidently, the only properties of the function $F$ we use are that it is Lipschitz
in both the left- and right-invariant metrics, and
satisfies a H\"older estimate $|D_sF(g) - D_sF(h)|\le C a^{-\alpha(g|h)}$ for all $s$
and all $g,h\in G$. Any such function on a hyperbolic group satisfies a central limit
theorem analogous to Corollary~\ref{distance_CLT}. For the sake of completeness,
therefore, we state this as a theorem:

\begin{theorem}[H\"older central limit theorem]\label{holder_CLT}
Let $G$ be a hyperbolic group, and $S$ a finite generating set for $G$.
Let $F$ be a real-valued function which is Lipschitz in both the left- and right-invariant
word metrics on $G$, and satisfies $|D_sF(g) - D_sF(h)|\le C a^{-\alpha(g|h)}$ for all $s$
in $S$ and all $g,h\in G$. Then there are constants $A$ and $\sigma$ so that if
$g_n$ is a random element of $G_n$ (in the $\nu$ measure), there is convergence in
probability
$n^{-1/2}(F(g_n) - An) \to N(0,\sigma)$.
\end{theorem}

\begin{remark}
The idea of using the thermodynamic formalism to study the relationship between distance
and word length in cocompact groups of isometries of hyperbolic space is due to
Pollicott--Sharp \cite{Pollicott_Sharp}; Corollary~\ref{distance_CLT} 
and Theorem~\ref{holder_CLT} above are simply the result of
combining their work with \cite{Coornaert} and \cite{Calegari_Fujiwara}. Nevertheless,
we believe they are new.
\end{remark}

\section{Random walks}

The main references for this section are Kaimanovich \cite{Kaimanovich} and
Kaimanovich--Vershik \cite{Kaimanovich_Vershik}.
The theory of random walks is a vast and deep subject, with connections to many different
parts of mathematics. Therefore it is necessary at a few points to appeal to some standard 
(but deep) results in probability theory, whose proof lies outside the scope of this survey.
A basic reference for probability theory is \cite{Stroock}. We give more specialized references
in the text where relevant.

This section is brief compared to the earlier sections, and is not meant to be comprehensive. 

\subsection{Random walk}

Let $G$ be a group and let $\mu$ be a probability measure on $G$. 
We further assume that $\mu$ is {\em nondegenerate}; i.e.\/ that the support of $\mu$ generates $G$ 
as a semigroup. An important example is the
case where $\mu$ is the uniform measure on a symmetric finite generating set $S$.
There are two ways to describe random walk on $G$ determined by $\mu$: as a sequence of
elements visited in the walk, or as a sequence of increments. In the first description,
a random walk $y:=\id,y_1,y_2,\cdots$ is a Markov chain with state space $G$,
with initial state $\id$, and with transition probability $p_{gh}=\mu(g^{-1}h)$. In the second
description, a random walk $z:=z_1,z_2,\cdots$ is a sequence of random elements of $G$
(the increments of the walk),
independently distributed according to $\mu$. The two descriptions are related by
taking $y_n = z_1z_2\cdots z_n$. We write this suggestively as $z=D y$ and $y=\Sigma z$.

We use the notation $(G^\N,\mu^\N)$ for the product probability space, 
and $(G^\N,\Pee)$ for the probability space of infinite
sequences with the measure $\Pee$ on cylinder sets defined by
$$\Pee(\lbrace y:y\text{ begins } \id,y_1,\cdots,y_n\rbrace) = p_{\id y_1}p_{y_1y_2}\cdots p_{y_{n-1}y_n}$$
With this notation, $z$ is a random element of $(G^\N,\mu^\N)$ and $y$ is a random element of
$(G^\N,\Pee)$.

The shift operator $T$ acts on $G^\N$ by $(Tz)_n = z_{n+1}$ or $(Ty)_n = y_{n+1}$. It is
measure preserving for $\mu^\N$ but not for $\Pee$; in fact, from the definition, the support of
$\Pee$ is contained in the set of sequences starting at $\id$. The action of the shift $T$ 
on $(G^\N,\mu^\N)$ is ergodic. For, if $A$ is a subset satisfying 
$A=T^{-1}(A)$, then a sequence $z$ is in $A$ if and only
if $T^n(z)$ is in $A$ for sufficiently big $n$. This is a tail event for the sequence of
independent random variables $z_i$, so
by Kolmogorov's $0$--$1$ law (see \cite{Stroock} Thm.~1.1.2) $A$ has measure $0$ or $1$.

\begin{definition}
Let $G$ be a group and $S$ a finite generating set. Let $\mu$ be a probability measure on $G$.
The {\em first moment} of $\mu$ is $\sum |g|\mu(g)$; if this is finite, we say $\mu$ has
{\em finite first moment}.
\end{definition}

\begin{lemma}\label{drift_exists}
Let $\mu$ be a probability measure on $G$ with finite first moment. Let $\id,y_1,y_2,\cdots$ 
be a random walk determined by $\mu$. Then $L:=\lim_{n\to \infty} |y_n|/n$ exists almost surely, and
is independent of $y$. In fact, if $\mu^{*n}$ denotes $n$-fold convolution
(i.e.\/ the distribution of the random variable $y_n$), 
then $L=\lim_{n\to\infty} \sum |g|\mu^{*n}(g)$.
\end{lemma}
\begin{proof}
We set $z=D y$. Define $h_n(z):=|y_n|$. Then $h_n$ satisfies
$$h_{n+m}(z) \le h_n(T^mz) + h_m(z)$$
i.e.\/ $h_n$ form a {\em subadditive cocycle}. Kingman's subadditive ergodic theorem
(see e.g.\/ \cite{Steele}) says that for any subadditive $L^1$ cocycle $h_n$ on a space
with a $T$-invariant measure, the limit $\lim_{n \to \infty} h_n(z)/n$
exists a.s. and is $T$-invariant. In our circumstance, finite first moment implies that $h_1$
(and all the $h_n$) are in $L^1$, so the theorem applies. Since the action of $T$ on $(G^\N,\mu^\N)$ 
is ergodic, the limit is independent of $z$. The lemma follows.
\end{proof}

$L$ as above is called the {\em drift} of the random walk associated to $\mu$.
Since each $|y_n|\ge 0$ we necessarily have $L\ge 0$. 

\begin{example}
If $G=\Z^n$ and $\mu$ is symmetric (i.e.\/ $\mu(g)=\mu(g^{-1})$ for all $g$) with finite
support, then $L=0$.
\end{example}

We now focus our attention on the case of hyperbolic groups and simple random walk (i.e.\/ when
$\mu$ is the uniform measure on a finite symmetric generating set). 

\begin{lemma}\label{positive_drift}
let $G$ be a nonelementary hyperbolic group, and let $\mu$ be a nondegenerate
probability measure on $G$ with finite first moment. 
Then the drift $L$ of random walk with respect to $\mu$ is positive.
\end{lemma}
\begin{proof}
We give the idea of a proof. Let $\mu^{*n}$ denote
the $n$-fold convolution of $\mu$ as before. The probability measures $\mu^{*n}$ have a
subsequence converging to a weak limit $\mu^{*\infty}$ in $\overline{G}$. Clearly
the support of $\mu^{*\infty}$ is contained in $\partial_\infty G$ (a group for which
$\limsup_{n\to\infty} \mu^{*n}(g)>0$ for any $g$ and for $\mu$ nondegenerate
is a finite group).

To prove the lemma it suffices to show that for any $C$, for sufficiently large enough $n$ 
there is an inequality $\sum_{g,h} (|hg| - |h|) \mu^{*n}(g)\mu^{*N}(h)\ge C>0$ for
all $N\ge n$, since then $L\ge C/n$. Now, for each $h$, if $g$ satisfies $|hg|-|h|<C$, then the closest point
on the geodesic from $h^{-1}$ to $g$ is within $\delta$ of some geodesic from $\id$ to $h^{-1}$.
So as $|g|$ goes to infinity, the $a$-distance from $h^{-1}$ to $g$ goes to $0$.
Hence for this inequality to fail to hold, almost half of the
mass of $\mu^{*n}\times\mu^{*N}$ must be concentrated near the {\em antidiagonal};
i.e.\/ the set of $(g,g^{-1}) \subset G \times G$.

From this we can deduce that either the desired inequality is satisfied, or else
most of the mass of $\mu^{*n}$ must be concentrated near a single geodesic through $\id$.
Taking $n \to \infty$, the support of $\mu^{*\infty}$ must consist of exactly two points, 
and $G$ is seen to be elementary, contrary to hypothesis.
\end{proof}

\begin{remark}
It is a theorem of Guivarc'h (see \cite{Woess}, Thm.~8.14) 
that if $G$ is any group with a nondegenerate measure
$\mu$ (always with finite first moment)
for which the drift of random walk is zero, then $G$ is amenable. Some care is required
to parse this statement: on an amenable group some nondegenerate measures {\em may} have positive drift,
but on a nonamenable group, {\em every} nondegenerate measure has positive drift.

A nonelementary hyperbolic group always contains many nonabelian free groups, and is therefore
nonamenable; this gives a more highbrow proof of Lemma~\ref{positive_drift}.
\end{remark}

\begin{lemma}[Kaimanovich \cite{Kaimanovich} 7.2]\label{regular_sequence}
Let $X$ be a $\delta$-hyperbolic space. The following two conditions are equivalent for a sequence
$x_n$ in $X$ and a number $L>0$:
\begin{enumerate}
\item{$d(x_n,x_{n+1})\le o(n)$ and $d(x_0,x_n)=nL+o(n)$;}
\item{there is a geodesic ray $\gamma$ so that $d(x_n,\gamma(Ln))=o(n)$.}
\end{enumerate}
\end{lemma}
A sequence $x_n$ satisfying either condition is said to be {\em regular}.
\begin{proof}
That (2) implies (1) is obvious, so we show that (1) implies (2).
For simplicity, we use the notation $|y|:=d(x_0,y)$.
The path obtained by concatenating geodesics from $x_n$ to $x_{n+1}$ has finite $a$-length,
and therefore converges to some unique $x_\infty \in \partial_\infty X$. 

Let $\gamma_n$ (resp. $\gamma_\infty$) be geodesic rays from the origin to $x_n$ (resp. $x_\infty$)
and parameterize them by distance from the origin. Fix some positive $\epsilon$, and
let $N=N(\epsilon)$ be such that for any two $n,m>N$ the geodesics $\gamma_n$ and $\gamma_m$ are within
$\delta$ on the interval of length $(L-\epsilon)n$, and let $p_n = \gamma_n((L-\epsilon)n)$ so
that $d(p_n,\gamma_m)\le \delta$ for $n>N$. Now, $d(p_{n-1},p_n)\le L-\epsilon + 4\delta$ and therefore
$d(p_n,p_m)\le |n-m|(L-\epsilon+4\delta)$. On the other hand, $d(p_n,p_m)\ge ||p_m|-|p_n||=|n-m|(L-\epsilon)$.
Consequently the sequence $p_i$ is a quasigeodesic, and therefore there is
a constant $H=H(\delta,L)$ so that $d(p_n,p_Nx_\infty)\le H$ for any $n\ge N$. Since $p_Nx_\infty$
and $\gamma_\infty$ are asymptotic, $d(p_n,\gamma_\infty)\le H+\delta$ for sufficiently large $n$,
and therefore $d(x_n,\gamma_\infty)\le H+\delta + (|x_n|-n(L-\epsilon))$ for sufficiently large $n$.
Taking $\epsilon \to 0$ proves the lemma with $\gamma=\gamma_\infty$.
\end{proof}

Together with Lemma~\ref{positive_drift} this gives the following Corollary:

\begin{corollary}[Kaimanovich \cite{Kaimanovich} 7.3]\label{random_walk_near_geodesic}
Let $G$ be a nonelementary hyperbolic group, and let $\mu$ be a nondegenerate probability 
measure on $G$ with finite first moment. Then there is $L>0$ so that for a.e. 
random walk $y$ there is a unique geodesic ray $\gamma_y$ with $d(y_n,\gamma_y(Ln))=o(n)$.
\end{corollary}
\begin{proof}
It suffices to show that if $\mu$ has finite first moment, then $d(y_n,y_{n+1})=o(n)$ almost
surely. Let $z=D y$, and for any $\epsilon > 0$ let $E_n$ be the event that $|z_n|\ge
\epsilon n$. Then the probability of $E_n$ is $\sum_{|g|\ge \epsilon n} \mu(g)$, and therefore
$$\sum_n \Pr(E_n) = \sum_n \sum_{|g|\ge \epsilon n} \mu(g) \le 
\frac 1 \epsilon \sum (|g|+1)\mu(g) < \infty$$
Therefore by the easy direction of the Borel--Cantelli lemma (see e.g.\/ \cite{Stroock} 1.1.4)
the probability that $E_n$ occurs infinitely often is zero. Since this is true for every
$\epsilon$, we have $d(y_n,y_{n+1}) = |z_n| =  o(n)$ almost surely.
\end{proof}

\subsection{Poisson boundary}

Define an equivalence relation $\sim$ on $G^\N$ by $y \sim y'$ if and only if there are
integers $k,k'$ so that $T^ky = T^{k'}y'$.

\begin{definition}
The {\em measurable envelope} of $\sim$ is the smallest measurable equivalence relation 
generated by $\sim$.
The quotient measure space $(\Gamma,\nu)$ of $(G^\N,\Pee)$ by the measurable envelope
is called the {\em Poisson boundary} of $G$ with respect to $\mu$.
\end{definition}

In other words, $\nu$-measurable functions on $\Gamma$ correspond precisely to $T$-invariant
$\Pee$-measurable functions on $G^\N$. We let $\bnd:G^\N \to \Gamma$ be the quotient map, so
that $\bnd \Pee = \nu$. 

Now, $G$ acts on $G^\N$ on the left coordinatewise. This action commutes with $T$, and
descends to an action on $\Gamma$. Since $\sim$ is $T$-invariant, $\bnd \Pee = \bnd T\Pee$,
so $\nu = \sum_g \mu(g)g\nu$; i.e.\/ the measure $\nu$ is {\em $\mu$-stationary}.

\begin{definition}
A {\em $\mu$-boundary} is a $G$-space with a $\mu$-stationary measure $\lambda$ which
is obtained as a $T$-equivariant (measurable) quotient of $(G^\N,\Pee)$.
\end{definition}

Any $\mu$-boundary factors through $(\Gamma,\nu)$. A $\mu$-boundary is {\em $\mu$-maximal} if
the map from $(\Gamma,\nu)$ is a measurable isomorphism. Kaimanovich \cite{Kaimanovich} gave
two very useful criteria for a $\mu$-boundary to be maximal. 

\begin{theorem}[Kaimanovich ray criterion \cite{Kaimanovich} Thm.~5.5]
Let $B$ be a $\mu$-boundary, and for $y\in G^\N$ let $\Pi(y) \in B$ be the image of
$y$ under the ($G$-equivariant) quotient map $\Pi:G^\N \to B$. 
If there is a family of measurable maps $\pi_n:B \to G$ such that $\Pee$-a.e. 
$d(y_n,\pi_n(\Pi(y)))=o(n)$ then $B$ is maximal.
\end{theorem}

Together with Corollary~\ref{random_walk_near_geodesic}, this gives the following 
important result:

\begin{corollary}[Kaimanovich \cite{Kaimanovich} Thm.~7.6]
Let $G$ be a nonelementary hyperbolic group, and let $\mu$ be a nondegenerate 
probability measure of finite first moment. Let $\Pi:G^\N \to \partial_\infty G$ take
a random walk to its endpoint (which exists $\Pee$-a.e.), and let $\lambda=\Pi \Pee$.
Then $(\partial_\infty G,\lambda)$ is the Poisson boundary of $G,\mu$.
\end{corollary}
\begin{proof}
Simply define $\pi_n$ to be the maps that take a point $y \in \partial_\infty G$ to 
$\gamma_y(nL)$ where $\gamma_y$ is a parameterized geodesic ray from $\id$ to $y$, and $L$
is the drift.
\end{proof}

\subsection{Harmonic functions}

\begin{definition}
If $f$ is a function on $G$, the operator $P_\mu$ (convolution with $\mu$) is defined by
$P_\mu f(g):=\sum_h f(gh) \mu(h)$. A function $f$ on $G$ is {\em $\mu$-harmonic} (or just
harmonic if $\mu$ is understood) if it is fixed by $P_\mu$; i.e.\/ if it
satisfies $f(g) = \sum_h f(gh) \mu(h)$ for all $g$ in $G$.
\end{definition}

In general we need to impose some condition on $f$ for $\sum_h f(gh) \mu(h)$ to be defined. 
If the support of $\mu$ is finite, then $f$ can be arbitrary, but if the support of $\mu$
is infinite, we usually (but not always!) require $f$ to be in $L^\infty$. 
We let $H^\infty(G,\mu)$ denote the Banach space of bounded $\mu$-harmonic functions on $G$.

In probabilistic terms, if $f$ is harmonic and $y\in G^\N$ is a random walk, the random
variables $f_n:=f(y_n)$ are a {\em martingale}; i.e.\/ the expected value of $f_n$
given $y_{n-1}$ is $f_{n-1}$ (see e.g.\/ \cite{Stroock} \S~5.2 for
an introduction to martingales). There is an intimate relation between harmonic
functions and Poisson boundaries, expressed in the following proposition.

\begin{lemma}
The Banach spaces $H^\infty(G,\mu)$ and $L^\infty(\Gamma,\nu)$ are isometric.
\end{lemma}
\begin{proof}
Given $f\in H^\infty(G,\mu)$ and $y\in G^\N$, the random variables $f(y_n)$ are a bounded martingale,
and therefore by the martingale convergence theorem (\cite{Stroock} Thm.~5.2.22),
converge a.s. to a well-defined limit. Evidently this limit
is measurable and $T$-invariant, and therefore descends to a function on $\Gamma$
which we denote $\widehat{f}$. Explicitly,  $\widehat{f}(\bnd y):=\lim_{n \to \infty} f(y_n)$.

Conversely, given $\widehat{f} \in L^\infty(\Gamma,\nu)$ we define 
$f(g) = \int_\Gamma \widehat{f} d(g_*\nu)$ (this expression is known as the {\em Poisson formula}).
Since $\nu$ is stationary, $f$ is harmonic.

The mean value property of harmonic functions implies that these maps are isometries,
since a harmonic function achieves its maximum on the boundary.
\end{proof}

Note that the Poisson formula is available for {\em any} $\mu$-boundary. That is, if $B$ is
a $G$-space with a $\mu$-stationary probability measure $\lambda$, and $\widehat{f}$ is
any element of $L^\infty(B,\lambda)$, then $f(g):=\int_B \widehat{f} d(g\lambda)$ is
a bounded harmonic function on $G$. If $\lambda$ is not invariant, $f$ is typically
nonconstant.

The remainder of this section is devoted to some miscellaneous 
applications of random walks to hyperbolic and other groups.

\subsection{Green metric}

There is a close resemblance between the measure $\nu$ and the Patterson--Sullivan
measures constructed in \S~\ref{Patterson_Sullivan_subsection}. 
This resemblance can be sharpened if one looks at a natural metric
on $G$ adapted to the random walk, namely the so-called {\em Green metric}.

\begin{definition}
Let $G$ be a group and $\mu$ a probability measure on $G$ with finite first moment.
The {\em Green metric} on $G$ is the metric 
for which the distance between $g$ and $h$
is $-\log$ of the probability that random walk starting at $g$ ever hits $h$.
\end{definition}

If $\mu$ is symmetric, so is the Green metric, since random walks are time-reversible.
Note that the Green metric is degenerate if random walk is recurrent. For simple
random walk, this occurs only if $G$ is finite, 
or is virtually $\Z$ or $\Z^2$, by a classical result of Varopoulos (see \cite{Varopoulos_etal}).
For nondegenerate measures with finite first moment on
non-elementary hyperbolic groups, Blach\`ere and Brofferio \cite{Blachere_Brofferio} 
show that the Green metric
and the word metric are quasi-isometric (one needs to be somewhat careful: the
Green metric is {\em not} in general a geodesic metric). 

\begin{theorem}[Blach\`ere--Ha\"issinsky--Mathieu \cite{bhm}, Thm.~1.3]\label{harmonic_measure_dimension}
Let $G$ be a non-elementary hyperbolic group, and for $y\in \partial_\infty G$,
let $B(y,R)$ denote the ball of radius $R$ in the $a$-metric 
(see Definition~\ref{a_distance}). Let $\mu$ be
a symmetric probability measure with finite first moment, and let $\nu$ be
the associated harmonic measure on $\partial_\infty G$. Then for $\nu$-almost every
$y\in \partial_\infty G$, there is convergence
$$\lim_{R \to 0} \log \nu(B(y,R))/\log R = \ell_G/aL$$
where $L$ is the drift in the word metric, and $\ell_G$ is the drift in the
Green metric.
\end{theorem}

Note that Kingman's subadditive ergodic theorem implies that
the drift $\ell_G$ with respect to the Green metric is well-defined, essentially
by the same argument as the proof of Lemma~\ref{drift_exists}.

\subsection{Harnack inequality}

The classical {\em Harnack inequality} relates the values of a positive harmonic
function at two points. In its infinitesimal version, it asserts an upper bound
on the logarithmic derivative of a positive harmonic function.

Let $f$ be a non-negative bounded harmonic function on $\H^n$, for simplicity.
The Poisson formula says that $f(p) = \int_{S^{n-1}_\infty} \widehat{f} d\nu_p$
where $\nu_p$ is the visual measure as seen from $p$. If $\nu$ is visual measure
as seen from the origin, and $g$ is any isometry taking the origin to $p$, then
$\nu_p=g_*\nu$. To understand how $f$ varies as a function of $p$ therefore, it
suffices to understand how $\nu_p$ varies as a function of $p$. If
$B$ is an infinitesimal ball centered at some point $y$ in $S^{n-1}_\infty$, 
then the visual size of $B$ grows like $e^{t(n-1)}$ as one moves distance $t$ in the
direction of $y$. Hence:

\begin{proposition}[Harnack inequality]
Let $f$ be a non-negative bounded harmonic function on $\H^n$. Then the logarithmic
derivative of $f$ satisfies the inequality $|d\log f| \le (n-1)$.
\end{proposition}

If $f$ is a non-negative harmonic function on a group $G$, the analog of this
inequality is $f(gs)/f(g) \le e^D$ for any $g\in G$ and $s\in S$
where $D$ is the dimension of $\nu$, which can be determined from 
Theorem~\ref{harmonic_measure_dimension}.

If $S$ is a closed surface of genus $\ge 2$ and $\rho:\pi_1(S) \to G$ is injective,
then $\pi_1(S)$ acts on $G$ by left translation, and there is an associated foliated
bundle with fiber the ideal circle $\partial_\infty \pi_1(S)$ with its natural
$\pi_1(S)$ action. We can build a harmonic connection for this circle bundle;
i.e.\/ a choice of measure $m_g$ on the circle $S^1(g)$ over each $g\in G$ so that for
any subset $A\subset S^1$ we have $m_g(A) = \sum \mu(s)m_{gs}(A)$. Since the circle
is $1$-dimensional, these measures integrate to {\em metrics} on the circles $S^1(g)$
for which the curvature is harmonic. The Harnack inequality then gives {\it a priori}
bounds on this curvature, and one can deduce local compactness results for families of
injective surface maps of variable genus. For stable minimal surfaces in hyperbolic
$3$-manifolds, such {\it a priori} bounds were obtained by Schoen \cite{Schoen} and
are an important tool in low-dimensional topology. The idea of using Harnack-type
inequalities to obtain curvature bounds is due to Thurston \cite{Thurston_circles}
(also see \cite{Calegari_foliations}, Example~4.6).

\subsection{Monotonicity}

A {\em norm} on a group is a non-negative function $\tau:G \to \R$ 
so that $\tau(gh) \le \tau(g) + \tau(h)$ for all $g,h\in G$.
A functor from groups to norms is {\em monotone} if $\tau_H(\phi(g)) \le \tau_G(g)$
for any $g\in G$ and $\phi:G \to H$.

If $\tau$ is a norm on $G$, and $\mu$ is a probability measure with finite first
moment, it makes sense to study the growth rate of $\tau$ under $\mu$-random walk on $G$.
If $G$ is finitely generated, one can study the growth rate of $\tau$ under all
simple random walks; if they all have the same growth rate, this rate is an invariant
of $G$. Since $\mu$ random walk on $G$ pushes forward to $\phi_*\mu$ random walk
on $\phi(G)=H$, the growth rate of a monotone family of norms cannot increase
under a homomorphism; thus if the growth rate of $\tau_G$ on $G$ is strictly smaller
than the growth rate of $\tau_H$ on $H$, there are strong constraints on the homomorphisms
from $G$ to $H$.

As an example, consider the {\em commutator length} $\cl$. For any group $G$ and
any $g$ in the commutator subgroup $[G,G]$, the commutator length $\cl(g)$ is
just the least number of commutators in $G$ whose
product is $g$ (for technical reasons, one usually studies a closely related quantity, namely the
{\em stable commutator length}; see e.g\/ \cite{Calegari_scl} for an introduction).

One of the main theorems of \cite{Calegari_Maher} is as follows:
\begin{theorem}[Calegari--Maher \cite{Calegari_Maher}]
Let $G$ be hyperbolic, and let $\mu$ be a nondegenerate
symmetric probability measure with finite first moment whose
support generates a nonelementary subgroup. There is a constant $C$ so that if 
$g_n$ is obtained by random
walk of length $n$, conditioned to lie in $[G,G]$, then 
$$C^{-1} n/\log(n) \le \cl(g_n) \le C n/\log(n)$$
with probability $1-O(C^{-n^c})$.
\end{theorem}
Said another way, commutator length grows like $n/\log(n)$ under random walk in a hyperbolic group.
Similar estimates on commutator length can be obtained for groups acting in a suitable way on
(not necessarily proper) hyperbolic spaces; the most important examples are mapping class groups 
and relatively hyperbolic groups.

As a corollary, if $H$ is any finitely generated group, and commutator length in $H$ grows
like $o(n/\log(n))$
for simple random walk (with respect to some generating set), then there are
no interesting homomorphisms from $H$ to any hyperbolic group $G$, 
and no interesting actions of $H$ on certain hyperbolic complexes.

\section{Acknowledgments}
I would like to thank Vadim Kaimanovich, Anders Karlsson, Joseph Maher, 
Curt McMullen, Richard Sharp, and Alden Walker. I would also like to thank the
anonymous referee for some useful comments. 
Danny Calegari was supported by NSF grant DMS 1005246.


\begin{thebibliography}{99}
\bibitem{Blachere_Brofferio}
  S. Blach\`ere and S. Brofferio,
  \emph{Internal diffusion limited aggregation on discrete groups having
  exponential growth},
  Probab. Theory Relat. Fields {\bf 137} (2007), 323-–343
\bibitem{bhm}
  S. Blach\`ere, P. Ha\"issinsky and P. Mathieu,
  \emph{Harmonic measures versus quasiconformal measures for hyperbolic groups},
    preprint, arXiv:0806.3915
\bibitem{Bowditch}
  B. Bowditch,
  \emph{A topological characterization of hyperbolic groups},
  Jour. AMS {\bf 11} (1998), no. 3, 643--667
\bibitem{Bowen}
  R. Bowen,
  \emph{Equilibrium states and the ergodic theory of Anosov diffeomorphisms},
  Lecture Notes in Mathematics {\bf 470}, Springer-Verlag 1975
\bibitem{Brooks}
  R. Brooks,
  \emph{Some remarks on bounded cohomology},
  Riemann surfaces and related topics: Proceedings of the 1978 Stony Brook
  Conference (SUNY Stony Brook NY 1978) Ann. Math. Studies {\bf 97},
  Princeton University Press, 1981, 53--63
\bibitem{Calegari_foliations}
  D. Calegari,
  \emph{Foliations and the geometry of $3$-manifolds},
  Oxford Mathematical Monographs. Oxford University Press, Oxford, 2007.
\bibitem{Calegari_scl}
  D. Calegari,
  \emph{scl},
  MSJ Memoirs, {\bf 20}. Mathematical Society of Japan, Tokyo, 2009.
\bibitem{Calegari_Fujiwara}
  D. Calegari and K. Fujiwara,
  \emph{Combable functions, quasimorphisms and the central limit theorem},
  Ergodic Theory Dynam. Systems {\bf 30} (2010), no. 5, 1343--1369
\bibitem{Calegari_Maher}
  D. Calegari and J. Maher,
  \emph{Statistics and compression of scl},
  preprint, arXiv:1008.4952
\bibitem{Cannon}
  J. Cannon,
  \emph{The combinatorial structure of cocompact discrete hyperbolic groups},
  Geom. Ded. {\bf 16} (1984), no. 2, 123--148
\bibitem{Coelho_Parry}
  Z. Coelho and W. Parry,
  \emph{Central limit asymptotics for shifts of finite type},
  Israel J. Math. {\bf 69} (1990), no. 2, 235--249
\bibitem{Coornaert}
  M. Coornaert,
  \emph{Mesures de Patterson--Sullivan sure le bord d'un espace hyperbolique au sens de Gromov},
  Pac. J. Math. {\bf 159} (1993), no. 2, 241--270
\bibitem{Coornaert_Delzant_Papadopoulos}
  M. Coornaert, T. Delzant and A. Papadopoulos,
  {\em G\'eom\'etrie et th\'eorie des groupes}, 
  Les groupes hyperboliques de Gromov, Springer-Verlag, Berlin, 1990.  
\bibitem{Coornaert_Papadopoulos}
  M. Coornaert and A. Papadopoulos,
  \emph{Symbolic dynamics and hyperbolic groups},
  Lecture Notes in Mathematics {\bf 1539}, Springer-Verlag 1993
\bibitem{Epstein_et_al}
  D. Epstein, J. Cannon, D. Holt, S. Levy, M. Paterson and W. Thurston,
  \emph{Word processing in groups},
  Jones and Bartlett, Boston, MA, 1992
\bibitem{Epstein_Fujiwara}
  D. Epstein and K. Fujiwara,
  \emph{The second bounded cohomology of word-hyperbolic groups},
  Topology {\bf 36} (1997), 1275--1289
\bibitem{Federer}
  H. Federer,
  \emph{Geometric measure theory},
  Grund. der math. Wiss. {\bf 153} Springer-Verlag New York, 1969
\bibitem{Flajolet_Sedgewick}
  P. Flajolet and R. Sedgewick,
  \emph{Analytic combinatorics},
  Cambridge University Press, Cambridge, 2009
\bibitem{Furman}
  A. Furman,
  \emph{Random walks on groups and random transformations},
  Handbook of dynamical systems, Vol. 1a, 931--1014, North-Holland, Amsterdam, 2002
\bibitem{Gromov_hyperbolic}
  M. Gromov,
  \emph{Hyperbolic groups},
  Essays in group theory, MSRI Publ. {\bf 8}, 75--263, Springer, New York, 1987
\bibitem{Hall_Heyde}
  P. Hall and C. Heyde,
  \emph{Martingale limit theory and its application},
  Academic Press, New York, 1980
\bibitem{Kaimanovich}
  V. Kaimanovich,
  \emph{The Poisson formula for groups with hyperbolic properties},
  Annals of Mathematics, {\bf 152} (2000), 659--692
\bibitem{Kaimanovich_Vershik}
  V. Kaimanovich and A. Vershik,
  \emph{Random walks on discrete groups: boundary and entropy},
  Ann. Probab. {\bf 11} (1983), no. 3, 457--490
\bibitem{Kakutani}
  S. Kakutani,
  \emph{Random ergodic theorems and Markoff processes with a stable distribution},
  Proceedings of the Second Berkeley Symposium on Mathematical Statistics and Probability,
  1950, University of California Press, Berkeley (1951), 247--261
\bibitem{Kapovich_Kleiner}
  M. Kapovich and B. Kleiner,
  \emph{Hyperbolic groups with low dimensional boundary},
  Ann. Sci. \'Ecole Norm. Sup. (4) {\bf 33} (2000), no. 5, 647--669
\bibitem{Kesten}
  H. Kesten,
  \emph{Full Banach means on countable groups},
  Math. Scand. {\bf 7} (1959), 146--156
\bibitem{Maclachlan_Reid}
  C. Maclachlan and A. Reid,
  \emph{The arithmetic of hyperbolic $3$-manifolds},
  Springer-Verlag GTM {\bf 219}. Springer-Verlag, New York, 2003.
\bibitem{Maskit}
  B. Maskit,
  \emph{Kleinian groups},
  Grund. der Math. Wiss. {\bf 287}. Springer-Verlag, Berlin, 1988.
\bibitem{Nicholls}
  P. Nicholls,
  \emph{The ergodic theory of discrete groups},
  LMS {\bf 143}, Cambridge University Press, Cambridge, 1989
\bibitem{Patterson}
  S. Patterson,
  \emph{The limit set of a Fuchsian group},
  Acta. Math. {\bf 136} (1976), 241--273
\bibitem{Perelman_1}
  G. Perelman,
  \emph{The entropy formula for the Ricci flow and its geometric applications},
  preprint arXiv:math/0211159
\bibitem{Perelman_2}
  G. Perelman,
  \emph{Ricci flow with surgery on three-manifolds},
  preprint arXiv:math/0303109
\bibitem{Picaud}
  J.-C. Picaud,
  \emph{Cohomologie born\'ee des surfaces et courants g\'eod\'esiques},
  Bull. Soc. Math. France {\bf 125} (1997), no. 1, 115--142
\bibitem{Pollicott_CRPF}
  M. Pollicott,
  \emph{A complex Ruelle-Perron-Frobenius theorem and two counterexamples},
  Ergodic Theory Dynam. Systems {\bf 4} (1984), no. 1, 135--146
\bibitem{Pollicott_Sharp}
  M. Pollicott and R. Sharp,
  \emph{Comparison theorems and orbit counting in hyperbolic geometry},
  Trans. AMS {\bf 350} (1998), no. 2, 473--499
\bibitem{Ruelle}
  D. Ruelle,
  \emph{Thermodynamic formalism},
  Addison-Wesley Publishing Co., Reading, Mass., 1978
\bibitem{Sharp}
  R. Sharp,
  \emph{Local limit theorems for free groups},
  Math. Ann. {\bf 321} (2001), 889--904
\bibitem{Schoen}
  R. Schoen,
  \emph{Estimates for stable minimal surfaces in three-dimensional manifolds},
  Seminar on minimal submanifolds, pp. 111--126, Ann. of Math. Stud. {\bf 103},
  Princeton Univ. Press, Princeton, N.J., 1983. 
\bibitem{Stallings_ends}
  J. Stallings,
  \emph{On torsion-free groups with infinitely many ends},
  Ann. Math. (2) {\bf 88} (1968), 312--334
\bibitem{Steele}
  M. Steele,
  \emph{Kingman's subadditive ergodic theorem},
  Ann. Inst. Henri Poincar\'e {\bf 25} (1989), no. 1, 93--98
\bibitem{Stroock}
  D. Stroock,
  \emph{Probability Theory, an analytic view},
  Cambridge University Press, Cambridge, 1993
\bibitem{Sullivan_ergodic}
  D. Sullivan,
  \emph{On the ergodic theory at infinity of an arbitrary discrete group of 
  hyperbolic motions},
  Riemann surfaces and related topics: Proceedings of the 1978 Stony Brook Conference 
  (State Univ. New York, Stony Brook, N.Y., 1978), pp. 465–-496, 
  Ann. of Math. Stud. {\bf 97}, Princeton Univ. Press, Princeton, N.J., 1981. 
\bibitem{Sullivan}
  D. Sullivan,
  \emph{The density at infinity of a discrete group of hyperbolic motions},
  Publ. Math. IHES {\bf 50} (1979), 171--202
\bibitem{Thurston_circles}
  W. Thurston,
  \emph{$3$-manifolds, foliations and circles II},
  preprint
\bibitem{Varopoulos_etal}
  N. Varopoulos, L. Saloff-Coste, and T. Coulhon,
  \emph{Analysis and geometry on groups},
  Cambridge Tracts in Math., {\bf 100}, Cambridge University Press, Cambridge, 1992.
\bibitem{Woess}
  W. Woess
  \emph{Random walks on infinite graphs and groups},
  Cambridge University Press, Cambridge, 2000
\end{thebibliography}
\end{document}